\documentclass[a4paper,12pt]{amsart}%
\usepackage{amsmath}
\usepackage{amsfonts}
\usepackage{amssymb}
\usepackage{graphicx}
\usepackage{comment}
\usepackage[margin=1.5 in]{geometry}
\usepackage{setspace}
\usepackage{marginnote}
\usepackage[utf8]{inputenc}
\usepackage{amsmath,amsthm, amssymb}
\usepackage{pdfrender,xcolor}
\usepackage{array}
\usepackage{hyperref}
\usepackage{enumerate}
\providecommand{\U}[1]{\protect\rule{.1in}{.1in}}
\providecommand{\U}[1]{\protect\rule{.1in}{.1in}}
\providecommand{\U}[1]{\protect\rule{.1in}{.1in}}
\providecommand{\U}[1]{\protect\rule{.1in}{.1in}}
\providecommand{\U}[1]{\protect\rule{.1in}{.1in}}
\newtheorem{theorem}{Theorem}

\newtheorem{corollary}[theorem]{Corollary}

\newtheorem{example}[theorem]{Example}

\newtheorem{lemma}[theorem]{Lemma}

\newtheorem{proposition}[theorem]{Proposition}
\newtheorem{remark}[theorem]{Remark}

\allowdisplaybreaks[1]
\newtheorem{thm}{Theorem}[section]

\newtheorem{defn}[thm]{Definition}

\theoremstyle{definition}
\numberwithin{equation}{section}

\newcommand{\C}{\mathbb C}

\newcommand{\Z}{\mathbb Z}

\newcommand{\Q}{\mathbb Q}
\newcommand{\E}{\mathbb E}

\renewcommand{\v}{\varphi}

\def \L{\Lambda}

\def\si{\sigma}

\newcommand{\La}{\Lambda}

\newcommand{\resumename}{R\'esum\'e}

\newcommand{\ssq}{\ensuremath{\subseteq}}
\newcommand*{\vertbar}{\rule[-1ex]{0.5pt}{2.5ex}}

\newcommand{\mc}[1]{\ensuremath{\mathcal{#1}}}
\newcommand{\set}[1]{{\left\{{#1}\right\}}}
\newcommand{\abs}[1]{{\left|{#1}\right|}}
\newcommand{\bra}[1]{{\left({#1}\right)}}
\newcommand{\norm}[1]{{\left\|{#1}\right\|}}
\newcommand{\scal}[1]{{\left\langle{#1}\right\rangle}}
\begin{document}
\title[Full spark frames]{Full Spark Frames in the Orbit of a Representation}
\author{Romanos Diogenes Malikiosis, Vignon Oussa}

\begin{abstract}
We present a new infinite family of full spark equal norm tight frames in finite dimensions arising from a unitary group representation,
where the underlying group is the semi-direct product of a cyclic group by a group of automorphisms. The only previously known
algebraically constructed infinite families were the harmonic, Gabor and Dihedral frames. Our construction hinges on a theorem that requires no group structure. Additionally, we illustrate our results by providing explicit constructions of full spark frames. 

\end{abstract}
\maketitle

\section{Introduction}

The notion of a full spark frame is of fundamental importance in the theory of signal processing and relevant in areas such as compressed sensing and quantum information theory. 
Although the notion of a frame was first defined and used in infinite-dimensional Hilbert spaces, its finite counterpart took
over in significance, especially in applications, and this is precisely the setting of the current article. Most notably,
we mention the solution to the Kadison-Singer problem \cite{KadisonSinger} and the SIC-POVM existence problem \cite{zauner}; the first problem can be formulated in
the finite frame setting and has implications to diverse areas of mathematics, such as operator theory and Banach space
theory \cite{Casazza13}, while the second
problem seems to indicate the existence of a profound connection between quantum information theory and algebraic number
theory, especially Hilbert's 12th problem \cite{SICANT}.
For a comprehensive
 introduction to finite frames, we refer the reader to the books \cite{fframes,waldron2018group}.

We will briefly describe the motivation behind the search for full spark frames, that arises from applications.
In a finite-dimensional complex vector space, say $\C^N$, a frame is just a finite spanning set. We consider the following
setting: suppose that we want to transmit a vector $v\in\C^N$, through a channel with erasures, i.e., some information packets
are lost to the recipient. We try to find an optimal way, to ensure that the recipient will have all the necessary information
to recover the initial vector $v$.

Transmitting the coordinates of $v$ is not a good strategy, because we cannot recover $v$ even if just one coordinate is lost. Similarly, transmitting each coordinate of $v$ a number of times (i.e., we send every coordinate to the recipient 100 times) is also not a good strategy. Recoverability of $v$ significantly increases when we send inner products $\scal{\v_j,v}$, where $\Phi=\set{\v_1,\dotsc,\v_M}\ssq\C^N$ is a spanning set of $\C^N$. We want to choose $\Phi$ in such a way, so that every $N$ vectors of $\Phi$ form a basis of $\C^N$. When this happens, we can recover $v$ from any $N$ inner products. For this reason, a set $\Phi$ with this property is also named \emph{maximally robust
to erasures}.
This means that no matter which $M-N$ inner products are lost, we can still recover $v$.

This is precisely the definition of a \emph{full spark frame}. Picking $M$ vectors at random (say, independently and uniformly from a Gaussian distribution) will yield
a full spark frame with probability 1, as long as $M>N$. However, \emph{verifying} whether a given set of vectors is full spark is an NP-hard problem \cite{ACM12}. Moreover, a random selection of vectors provides absolutely no control over other characteristics of a frame, such as tightness, coherence, condition number, uniformity of norm, etc. For this purpose, it is important
to build such frames algebraically (i.e., in a deterministic way), with prescribed properties.

On the one hand, observe that any frame is canonically associated with a tight frame obtained by applying the inverse of the square-root of the frame operator to it. Additionally, full spark frames exist in abundance \cite{ACM12,MR3633768}, and as such, full spark tight frames can be easily constructed. On the other hand, full spark equal norm tight frames are much harder to come by, and representation theory provides powerful tools that can be exploited to construct such frames in a systematic fashion. From a representation theoretic perspective, the following question arises naturally. Is it possible to characterize finite-dimensional and unitary representations of finite groups for which, there exists a vector whose orbit contains a full spark equal norm and tight frame? To the best of our knowledge, this is a difficult question which has not been systematically studied in the literature. However, several concrete representations have been thoroughly investigated. For instance \cite{MR2190681}, Lawrence, Pfander, and Walnut considered the action of a unitary and irreducible representation of a finite Heisenberg group acting on an $n$-dimensional vector space $n\geq 2$, when the underlying group is cyclic. Any frame contained in the orbit of a fixed vector is called a finite-dimensional Gabor frame. Such a representation is a finite-dimensional analog of the Schr\" odinger representations, which are known to be the group-theoretic foundation of Gabor Theory \cite{MR3495345,MR1843717}. 
Regarding the problem of existence of finite-dimensional Gabor tight frames, which are full spark, Lawrence, Pfander, and Walnut gave some partial solution in \cite{MR2190681}. Precisely, they proved that in the case where $n$ is prime, there exists an orbit containing a full spark frame consisting of $n^2$ vectors. In fact, the set of all vectors generating such a frame is Zariski open and dense in $\mathbb{C}^n.$ This problem was eventually completely settled by the first author of the current paper (see \cite{M15}) who not only established the existence of tight Gabor frames for all natural numbers $n\geq 2$ but also gave a deterministic construction. 
When the underlying group of the Heisenberg group is abelian, non-cyclic, the first author showed that the Gabor frames under question can never be full spark \cite{SparkDef}.

In another direction, the second author of this paper and Sheehan considered a class of quasiregular representations of the Dihedral group and investigated the existence and explicit construction of full spark frames \cite{oussa2015dihedral,oussa2018dihedral} associated with such representations. Surprisingly, their investigation revealed that the existence of full spark frames is a function of the parity of the dimension of the representations. Precisely, it is shown in \cite{oussa2018dihedral} that there exists an orbit, which is an equal norm full spark frame if and only if the dimension of the quasiregular representation is odd.  Additionally, an algorithmic construction of such frames can be found in  \cite{oussa2018dihedral}.

The present paper is organized as follows: in Section \ref{Char}, we define representations of finite groups that are full spark or spark deficient, and give some criteria for spark
deficient groups. Next, in Section \ref{Semi}, we focus on representations of semi-direct products of cyclic groups by a subgroup of automorphisms, 
and provide some new infinite families of
full spark frames, via the existence of full spark groups. The main technique generalizes the method used
in a previous article of the first-named author \cite{M15}. Finally, in Section \ref{nogroup} we extend these results by noting that an underlying group structure is not always
necessary to produce full spark tight frames, and provide some examples.

\section{Notation, terminology, and general results}\label{Notation_Terminology}

Let $\mathcal{H}$ be a finite-dimensional Hilbert space equipped with an inner
product which we shall denote by $\left\langle \cdot,\cdot\right\rangle .$ Recall that a
finite sequence of vectors ${\left(  {v}_{k}\right)  }_{k\in I}$ is called a
frame for $\mathcal{H}$ if there exist positive real numbers $A\leq B$ such
that
\[
A\left\Vert w\right\Vert ^{2}\leq%
{\displaystyle\sum\limits_{k\in I}}
\left\vert \left\langle w,v_{k}\right\rangle \right\vert ^{2}\leq B\left\Vert
w\right\Vert ^{2}%
\]
for all vectors $w\in\mathcal{H}$. The constants $A,B$ are called the lower
frame bound and upper frame bound of ${\left(  {v}_{k}\right)  }_{k\in
I}$ respectively. In cases where the two bounds coincide, we call ${\left(
{v}_{k}\right)  }_{k\in I}$ a \emph{tight frame}. Thanks to Cauchy-Schwarz
inequality, it is clear that any finite collection of vectors has an
upper frame bound. In fact, since we are only considering finite-dimensional
vector spaces, it is not difficult to verify that a collection of vectors
${\left(  {v}_{k}\right)  }_{k\in I}$ has a lower frame bound if and only if it spans $\mathcal{H}.$

Any frame ${\left(  {v}_{k}\right)  }_{k\in I}$ for $\mathcal{H}$ gives rise
to three important operators: the \emph{analysis}, \emph{synthesis} and \emph{frame operators}. The
analysis operator is given by the map
\[
C:w\mapsto\left(  \left\langle w,v_{k}\right\rangle \right)  _{k\in I}\in
l^{2}\left(  I\right), w\in \mathcal{H}.
\]
Its adjoint is the synthesis operator, $C^{\ast}:\left(  c_{k}\right)  _{k\in I}\mapsto%
{\displaystyle\sum\limits_{k\in I}}
c_{k}v_{k}\in\mathcal{H},$ and the operator
\[
S=C^{\ast}C:w\mapsto%
{\displaystyle\sum\limits_{k\in I}}
\left\langle w,v_{k}\right\rangle v_{k}%
\]
obtained by composing the analysis and the synthesis operators defines a self-adjoint map on $\mathcal{H}$ known as the frame operator. It is
a standard fact in frame theory that ${\left(  {v}_{k}\right)  }_{k\in I}$ is
a frame if and only if the frame operator $S$ is invertible \cite{fframes,waldron2018group}. Note also that
$S$ is invertible if and only if the analysis operator $C$ is injective. Moreover, if ${\left(  {v}_{k}\right)  }_{k\in I}$ is a frame for $\mathcal{H},$ then every vector $w\in\mathcal{H}$ admits admits an expansion of the type
\begin{equation}
w=%
{\displaystyle\sum\limits_{k\in I}}
\left\langle w,S^{-1}v_{k}\right\rangle v_{k}=%
{\displaystyle\sum\limits_{k\in I}}
\left\langle w,v_{k}\right\rangle S^{-1}v_{k}.\label{expansion}%
\end{equation}
Generally, since (\ref{expansion}) requires inverting the frame operator, the process of reconstructing a vector $w$ from the coefficients $\langle w, v_k \rangle, k\in I$ can be computationally expensive. However, for the
special case where ${\left(  {v}_{k}\right)  }_{k\in I}$ is a tight frame, the
frame operator is just a multiple of the identity map, and in that case (\ref{expansion}) simply becomes $$
w=%
\frac{1}{A}{\displaystyle\sum\limits_{k\in I}}
\left\langle w,v_{k}\right\rangle v_{k}.$$ For this reason, tight frames are more appealing than non-tight ones. 
\subsection{Unitary representations of finite groups}
Let $G$ be a finite group. A \emph{unitary representation} $\pi$ is a homomorphism
which maps $G$ into the group $\mathcal{U}\left(  \mathcal{H}_{\pi}\right)  $
of unitary operators acting in some Hilbert space $\mathcal{H}_{\pi}$. We say
that $\pi$ is an $n$-dimensional representation of $G$ if $\mathcal{H}_{\pi}$
is an $n$-dimensional vector space over the field of complex numbers. \\

A representation is said to be \emph{faithful} if its kernel is trivial. Moreover, two unitary representations $\pi$ and $\tau$ of $G$ acting on $\mathcal{H}%
_{\pi}$ and $\mathcal{H}_{\tau}$ respectively, are \emph{unitarily equivalent} if there exists a unitary operator $J:\mathcal{H}_{\pi}\rightarrow
\mathcal{H}_{\tau}$ such that  $J\pi\left(  x\right)  J^{-1}=\tau\left(
x\right)  $ for all $x\in G.$ \\

Let $\mathcal{K}$ be a vector subspace of $\mathcal{H}_{\pi}.$ We say that
$\mathcal{K}$ is $\pi$-invariant if for every $x\in G,$ $\pi\left(  x\right)
\mathcal{K}$ is a subset of $\mathcal{K}$. Associated to any $\pi$-invariant
subspace of $\mathcal{H}_{\pi},$ is a unitary representation $\pi
_{\mathcal{K}}$ of $G$ obtained by restricting the action of $\pi$ to
$\mathcal{K}$ as follows: $\pi_{\mathcal{K}}\left(  x\right)  =\pi\left(
x\right)  |\mathcal{K}$ for all $x\in G.$ Such a representation is called a
\emph{subrepresentation} of $\pi.$ Any representation which admits a nontrivial
subrepresentation is called \emph{reducible}. Otherwise, such a
representation is called \emph{irreducible}. Irreducible representations of finite
groups play a central role in their representation theory since they form the basic 
building blocks of all unitary representations. For this reason, an important
aspect of representation theory of finite groups consists of classifying up to
unitary equivalence, all irreducible representations of a given finite group.\\ 

A representation $\pi$ of a finite group $G$ acting in a finite-dimensional Hilbert space $\mathcal{H}_{\pi}$ is said to be \emph{cyclic} if there exists at least one vector $v\in \mathcal{H}_{\pi}$ such that $\pi(G)v$ spans  $\mathcal{H}_{\pi}.$ \\

The left action of any group on itself, induces an important unitary
representation known as the \emph{left regular representation} $L$ of $G$ which acts
on the Hilbert space $l^{2}\left(  G\right)  $ as follows. Given $f\in
l^{2}\left(  G\right)  ,$ and $x,y\in G,$ $L\left(  x\right)  f\left(
y\right)  =f\left(  x^{-1}y\right).$ The regular representation of a finite group is known to be the mother of all representations since any irreducible representation of $G$ is unitarily equivalent to a subrepresentation of the left regular representation (see \cite[Chapter 3]{MR3025357}). 

\subsection{Group full spark frames}

Let $\pi$ be unitary representation of a finite group acting in a finite-dimensional vector space. Then $\pi$ admits a cyclic vector (a vector whose orbit is a spanning set), if and only if $\pi$ is equivalent to a subrepresentation of the regular representation \cite{waldron2018group}. Thus, the study of representations giving rise to frames can be essentially reduced to the study of subrepresentations of the regular representation. 

\begin{defn}
Let $G$ be a finite group and let $\pi:G\rightarrow\mathcal{U}(\mathcal{H})$
be a unitary and cyclic representation acting on some finite-dimensional Hilbert space $\mathcal{H}$. We say that the representation $\pi$ is spark deficient if for every vector $v\in\mathcal{H},$ the sequence ${\left(  {\pi(g)v}\right)
}_{g\in G}$ is spark deficient (it is not a full spark frame.) Moreover, we say that $\pi$ is full
spark if there is a vector $v\in\mathcal{H}$ such that ${\left(  {\pi(g)v}\right)
}_{g\in G}$ is a full spark frame. Such a vector when it exists will be called a $\pi$-full spark vector.
\end{defn}

It is worth noting if a representation is full spark then the set of $\pi$-full spark vectors is a Zariski open and dense subset of the Hilbert space on which the representation is acting. \\


\begin{defn}
If $G$ is a finite group such that every irreducible unitary representation of
$G$ which is not a character is spark deficient, we call $G$ an irreducibly spark deficient
group. On the other hand, we shall say that $G$ is irreducibly full spark if every unitary
irreducible representation of $G$ is full spark.
\end{defn}

\begin{remark} Our investigation seems to reveal that non abelian irreducibly full spark groups are quite rare. Moreover, it is worth highlighting that there do exist groups which are neither irreducibly full spark nor irreducibly spark deficient. For instance, let $D_{2N}$ be the Dihedral group of order $2N, N\geq 3.$ Then it is known that \cite{oussa2015dihedral} such a group is irreducibly full spark if $N$ is prime, irreducibly spark deficient if $N$ is even, and neither if $N$ is odd and composite. Furthermore, since every abelian group is trivially irreducibly full spark, the classification of irreducibly full spark groups is only interesting for nonabelian groups. 
\end{remark}

\noindent Recall that for a finite group $G,$ the convolution $\ast$ product is
defined as follows. Given complex-valued functions $f,g$ defined on $G,$
\[
f\ast g\left(  x\right)  ={\sum\limits_{y\in G}}f\left(  y\right)  g\left(
y^{-1}x\right)  .
\]
Let $\pi$ be a unitary irreducible representation of $G$ acting in some
Hilbert space $\mathcal{H}$. Next, for a fixed nonzero vector $\phi
\in\mathcal{H}$, let $V_{\phi}:\mathcal{H}\rightarrow l^{2}\left(  G\right)  $
be a linear operator defined as follows: $V_{\phi}\psi\left(  x\right)
=\left\langle \psi,\pi\left(  x\right)  \phi\right\rangle $ for $\psi
\in\mathcal{H}.$ Put
\[
\mathcal{C}_{\pi}=\left\{  c:G\rightarrow\mathbb{C}:\abs{\mathrm{supp}\left(  c\right)}  =\dim\left(  \mathcal{H}
\right)
\right\}  .
\]

Our first result in this section, gives a complete characterization of $\pi$-full spark frame vectors in terms of the convolution product on $G.$ Precisely, it states that $\phi$ is a $\pi$-full spark vector if and only  if no element of $\mathcal{C}_{\pi}$ is a left zero divisor of $V_{\phi}\phi$ with respect to the convolution operation on $G.$

\begin{proposition}
\label{fullspark_char} $\phi$ is a $\pi$-full spark vector if and only if for
any $c\in\mathcal{C}_{\pi}$,  $ c\ast V_{\phi}\phi$ is a nonzero sequence. 
\end{proposition}

\begin{proof}
Suppose that there exists a vector $f$ in the
null-space of $V_{\phi}.$ That is, $\left\langle f,\pi\left(  x\right)
\phi\right\rangle =0\text{ for all }x\in G.$ Since $\pi$ is irreducible, $f$
must be the zero vector. Thus, $V_{\phi}$ is an injective linear operator and
consequently, $\left(  \pi\left(  x\right)  \phi\right)  _{x\in G}$ is a frame
for the Hilbert space $\mathcal{H}$. Moreover, the frame operator $V_{\phi
}^{\ast}V_{\phi}:\mathcal{H\rightarrow H}$ is defined as follows. Given
$\psi\in l^{2}(G),$ we have $V_{\phi}^{\ast}V_{\phi}\psi=\sum_{x\in G}\left\langle \psi,\pi\left(
x\right)  \phi\right\rangle \pi\left(  x\right)  \phi.$ Furthermore, for arbitrary $y\in G,$ it is clear that $V_{\phi}^{\ast}V_{\phi}\pi\left(  y\right)  \psi=\pi\left(  y\right)  V_{\phi}^{\ast}V_{\phi}\psi.$ Thus, the linear operator $V_{\phi}^{\ast}V_{\phi}$ is in the commutant of the
algebra generated by $\pi.$ Since $\pi$ is irreducible, according to Schur's
lemma, the linear operator $V_{\phi}^{\ast}V_{\phi}$ must be a constant
multiple of the identity map and consequently, $\left(  \pi\left(  x\right)
\phi\right)  _{x\in G}$ must be a tight frame for $\mathcal{H}$. This implies
that $V_{\phi}$ defines (up to multiplication by a constant) an isometry
between $\mathcal{H}$ and its image $V_{\phi}\left(  \mathcal{H}\right)  .$
Observing that $V_{\phi}\left(  \mathcal{H}\right)  $ is a left-invariant
vector space and letting $\lambda$ be the restriction of the left regular
representation of $G$ to $V_{\phi}\left(  \mathcal{H}\right)  ,$ it follows
that $V_{\phi}$ intertwines $\pi$ with $\lambda.$ As such, it is clear that
$\left(  \pi\left(  x\right)  \phi\right)  _{x\in G}$ is a full spark tight
frame for $\mathcal{H}$ if and only if $\left(  \lambda\left(  x\right)
V_{\phi}\phi\right)  _{x\in G}$ is a full spark tight frame for the Hilbert
space $V_{\phi}\left(  \mathcal{H}\right)  .$ However, the latter statement
holds if for arbitrary $c\in\mathcal{C}_{\pi},$ $
0\neq\sum_{x\in G}c\left(  x\right)  \lambda\left(  x\right)  V_{\phi}\phi.$ 
The stated result follows then from the observation that for arbitrary $y\in
G,$ $
\sum_{x\in G}c\left(  x\right)  V_{\phi}\phi\left(  x^{-1}y\right)  =c\ast
V_{\phi}\phi\left(  y\right).$

\end{proof}

\begin{proposition}
Let $\pi$ be a finite-dimensional representation of a finite group $G$ acting
in a Hilbert space $\mathcal{H}.$ Then $\pi$ is full spark with a $\pi$-full spark vector $\phi$ if and
only if for any nonzero vector $\psi\in\mathcal{H},$ the sequence $\left(
\left\langle \psi,\pi\left(  x\right)  \phi\right\rangle \right)  _{x\in G}$
has at most $\dim\left(  \mathcal{H}\right)  -1$ components which are equal to zero.
\end{proposition}

\begin{proof}
Let us suppose that $\pi$ is full spark with full spark vector $\phi.$ For any
subset $X$ of $G$ of cardinality $\dim\left(  \mathcal{H}\right)  ,$ let
$V_{\phi,X}:\psi\mapsto\left(  \left\langle \psi,\pi\left(  x\right)
\phi\right\rangle \right)  _{x\in X}$ be the analysis operator corresponding
to the frame $\left(  \pi\left(  x\right)  \phi\right)  _{x\in X}.$ Since for
any nonzero vector $\psi\in\mathcal{H}$ the sequence $\left(  \left\langle
\psi,\pi\left(  x\right)  \phi\right\rangle \right)  _{x\in X}$ is nonzero in
$l^{2}\left(  X\right)  ,$ then $\ \left(  \left\langle \psi
,\pi\left(  x\right)  \phi\right\rangle \right)  _{x\in G}$ has at most
$\dim\left(  \mathcal{H}\right)  -1$ components which are equal to zero.
Since the converse of the previous statement follows from similar arguments, we omit its proof.  
\end{proof}

\section{Spark deficient representations}\label{Char}

Throughout this section, unless we state otherwise, all representations under
consideration are assumed not to be unitary characters.\\

The following facts are straightforward yet essential observations. \\

\vskip 0.2cm \noindent (a) Let $\pi$ be a unitary irreducible representation of $G.$ If there
exists a finite subset $X$ of $G$ of cardinality at most $\dim{\mathcal{H}}$
such that the collection of unitary operators $\pi(x),x\in X$ is linearly
dependent in the vector space of linear operators acting on $\mathcal{H}$ then
$\pi$ is spark deficient.

\vskip 0.2cm \noindent (b) Let $\pi$ be a unitary representation of a finite group $G$ acting in a finite-dimensional Hilbert space $\mathcal{H}$ such that $\left\vert
G\right\vert \geq\dim\left(  \mathcal{H}\right) >1.$ If $\pi$ is not faithful then it must be spark deficient. Indeed, if $\pi$ is
assumed not to be faithful, then there exists a non-trivial element $x$ in its
kernel, which maps to the identity operator acting on $\mathcal{H}.$ Letting
$e$ be the identity in $G,$ we obtain that $\pi\left(  e\right)  -\pi\left(
x\right)  $ is the zero operator.

\subsection{The spark deficiency of the direct sum of two equivalent full spark representations}

Since irreducible representations are the building blocks of all representations, the following question is quite natural.  Given a full spark irreducible representation, is the representation obtained as the direct sum of copies of the given representation full spark as well? The result below answers this question negatively under some natural assumptions. 

\begin{proposition}
\label{multiplicity}Let $G$ be a finite group and let $\pi$ be a faithful
unitary representation of $G$ acting in $\mathcal{H}$ for which there exists
an element $x\in G$ such that the order of $\pi\left(  x\right)  $ is at least
twice of the dimension of $\mathcal{H}$. Then $\pi\oplus\pi$ must be spark deficient.
\end{proposition}

\begin{proof}
Suppose by contradiction that $\tau=\pi\oplus\pi$ is full spark. Then there
exists $\phi\in\mathcal{K=H}\oplus\mathcal{H}$ such that ${\left(  {\tau
(g)}\phi\right)  }_{g\in G}$ is a full spark frame for $\mathcal{K}$. Next,
let $x\in G$ such that the order of $\pi\left(  x\right)  $ is at least twice
of the dimension of $\mathcal{H}$. Since, every eigenvalue of $\tau\left(
x\right)  $ occurs at least twice, there exists an invertible matrix $Q$ such that $D=Q\tau\left(  x\right)  Q^{-1}$ is a diagonal matrix with diagonal entries listed in order as follows (including multiplicity): $\lambda,\lambda, \lambda_1,\cdots,\lambda_{N-2},$ such that $\left\vert \lambda\right\vert =\left\vert \lambda_{1}\right\vert
=\cdots=\left\vert \lambda_{N-2}\right\vert =1.$ Let $e_1,\cdots,e_N$ be the canonical orthonormal basis for $\mathcal{K}=\mathbb{C}^{N}.$ By assumption, the restriction of $\tau$ to the subgroup generated by $x$ is a cyclic representation  of the subgroup and consequently, it is clear that the analysis operator $f\mapsto\left(  \left\langle f,\tau\left(  y\right)  \phi\right\rangle
\right)  _{y\in\left\langle x\right\rangle }$ defines an injection between $\mathcal{K}$ and $l^{2}\left(  \left\langle
x\right\rangle \right).$ Next, for any integer $k,$ straightforward calculations
give
\begin{align*}
\left\langle f,\tau\left(  x^{k}\right)  \phi\right\rangle =\left\langle
Q^{\ast}\left(  Q^{\ast}\right)  ^{-1}f,\tau\left(  x^{k}\right)
\phi\right\rangle =\left\langle \left(  Q^{\ast}\right)  ^{-1}f,D^{k}Q\phi\right\rangle .
\end{align*}
Moving forward, we shall assume that $\phi$ satisfies the following additional conditions: $Q\phi=\alpha_{1}e_{1}+\alpha_{2}e_{2}+w$ and $w\in\mathbb{C}$-span$\left\{  e_{3},\cdots,e_{N}\right\}  .$ Thus,
\begin{align*}
\left\langle f,\tau\left(  x^{k}\right)  \phi\right\rangle  =\lambda^{-k}\cdot\left\langle \left(  Q^{\ast}\right)  ^{-1}f,\alpha
_{1}e_{1}+\alpha_{2}e_{2}\right\rangle +\left\langle \left(  Q^{\ast}\right)
^{-1}f,D^{k}w\right\rangle .
\end{align*}
Assuming additionally that $f$ is a nonzero vector in $\mathcal{K}$ such that $\left(  Q^{\ast
}\right)  ^{-1}f\in\mathbb{C}e_{1}\oplus\mathbb{C}e_{2}$ and $\left(  Q^{\ast
}\right)  ^{-1}f$ is orthogonal to $\alpha_1 e_1+\alpha_2 e_2,$ it follows that
\[
\left\langle f,\pi\left(  x^{k}\right)  \phi\right\rangle =\overline
{\lambda^{k}}\cdot\left\langle \left(  Q^{\ast}\right)  ^{-1}f,\alpha_{1}%
e_{1}+\alpha_{2}e_{2}\right\rangle =0
\]
for all integers $k.$ Thus, $\left(  \left\langle
f,\tau\left(  y\right)  \phi\right\rangle \right)  _{y\in\left\langle
x\right\rangle }$ is the zero sequence and this clearly contradicts the fact
that the linear map $f\mapsto\left(  \left\langle f,\pi\left(  h\right)  \phi\right\rangle
\right)  _{h\in\left\langle x\right\rangle }$ is assumed to be injective. 
\end{proof}

It is perhaps worth noting that the arguments in the proof of Proposition \ref{multiplicity} also establish the following. 

\begin{proposition}
Let $G$ be a finite group and let $\pi$ be a faithful unitary representation
of $G$ acting in $\mathcal{H}$ for which there exists an element $x\in G$ such
that the order of $\pi\left(  x\right)  $ is at least equal to the dimension
of $\mathcal{H}$ and $\pi\left(  x\right)  $ has an eigenvalue occurring with
multiplicity strictly greater than one. Then $\pi$ must be spark deficient. 
\end{proposition}

\subsection{Groups with non-trivial centers are irreducibly spark deficient}

Recall that $G$ is a $p$-group if each element in $G$ has a power of $p$ as
its order.

\begin{proposition}
\label{center} Let $G$ be a non-abelian finite group. If $G$ has a non-trivial
center then $G$ is irreducibly spark deficient. Moreover, the following groups are all irreducibly spark deficient: $p$-groups and nilpotent groups.
\end{proposition}

\begin{proof}
Let $\pi$ be a unitary irreducible representation of $G$ acting in some
Hilbert space $\mathcal{H}$ such that $\dim\mathcal{H}>1.$ Let $Z\left(
G\right)  $ be the center of $G.$ By assumption, $Z\left(  G\right)  $ is a
non-trivial subgroup of $G.$ Next, appealing to Schur's lemma, for arbitrary
$z\in$ $Z\left(  G\right)  $, $\pi\left(  z\right)  $ is a scalar multiple of
the identity operator acting on $\mathcal{H}$. Moreover, given any nonzero
vector $f\in\mathcal{H}$, the span of $\pi\left(  Z\left(  G\right)  \right)  f$ is a
one-dimensional vector subspace of $\mathcal{H}$ and consequently, $\pi$ is
spark deficient. The second claim of the result follows from the well-known fact that  $p$-groups and (more generally) nilpotent groups have non-trivial centers.  
\end{proof}

\begin{remark}
For a unitary representation $\pi$  of a finite group, the following statements are not mutually exclusive. (a) $\pi$ is spark deficient and (b) $\pi$ has an orbit which properly contains a full spark frame. For instance, it is known that the (non unitary characters) irreducible representations of the finite Heisenberg group (which is a nilpotent group) enjoy both of these properties \cite{M15}. In fact, in the setting of groups with non-trivial centers, the more general concept of projective representation is more appropriate. 
\end{remark}

\begin{proposition}
\label{semidirect_general} Let $G=K\rtimes H$ (where $H$ is a non-trivial
subgroup of the automorphism group of $K$) be a finite semi-direct product
group equipped with the operation
\[
\left(  n,h\right)  \left(  m,s\right)  =\left(  nh\left(  m\right)
,hs\right)  \text{ for }\left(  n,h\right)  ,\left(  m,s\right)  \in K\times
H.
\]
If the center of $K$ has a non-trivial element which is fixed by the action of
$H$ then $G$ is irreducibly spark deficient.
\end{proposition}

\begin{proof}
Let $e$ be the identity element in $H$ and let $n$ be a non-trivial element of the center of $N$ which is fixed by the
action of the automorphism group $H.$ For arbitrary $\left(  m,h\right)  \in
G,$
\[
\left(  n,e\right)  \left(  m,h\right)  \left(  n^{-1},e\right)  =\left(
nm,h\right)  \left(  n^{-1},e\right)  =\left(  nmh\left(  n^{-1}\right)
,h\right)  .
\]
Since $h\left(  n^{-1}\right)  =h^{-1}\left(  n\right)  =n,$ we obtain $\left(  n,e\right)  \left(  m,h\right)  \left(  n^{-1},e\right)  =\left(
nmn^{-1},h\right).$ Now, since $n$ commutes with $m,$ we have $\left(  n,e\right)  \left(  m,h\right)  \left(  n^{-1},e\right)  =\left(
mnn^{-1},h\right)  =\left(  m,h\right).$ Consequently, $\left(  n,e\right)  $ is a non-trivial element of the center of
$G.$
\end{proof}

A straightforward application of Proposition \ref{semidirect_general} gives the following. 


\begin{corollary}
Let $G=\mathbb{Z}_{N}\rtimes H$ where $H$ is a cyclic group generated by
$h\in\mathbb{Z}_{N}$ such that $h,N$ are relatively prime and $h>1.$ Then $G$ has a
non-trivial center if and only if there exists a nonzero element $x$ in
$\mathbb{Z}_{N}$ which is a zero divisor satisfying the equation $(h-1)x=0.$ 
In other words, if $h$ does not act freely on all of the nonzero elements of
$\mathbb{Z}_{N}$, then $G$ is necessarily irreducibly spark deficient. 
\end{corollary}

Note that if $G=\mathbb{Z}_{N}\rtimes H,$ $H$ is non-trivial and $N$ is prime,
then $G$ has a trivial center. In fact, we will prove in the subsequent
section that when $N$ is prime then $G=\mathbb{Z}_{N}\rtimes H$ is irreducibly full spark
(see Theorem \ref{ZnH}.)

\section{\texorpdfstring{Full-spark frames arising from irreducible representations of 	$\mathbb{Z}_{N}\rtimes H$}{}}\label{Semi}

\subsection{A summary of our main results}

While the preceding section focuses on conditions under which a group fails to be irreducibly full spark, we shall provide here, some positive results for a class of semi-direct product groups. Our main aim in this section is to establish the following.

\begin{theorem}
	\label{ZnH} For every prime natural number $N$ and every subgroup $H<\Z_N^{\times}$, the semi-direct product group $\mathbb{Z}_N\rtimes H$ is irreducibly full spark. 
\end{theorem}

The following shows that Theorem \ref{ZnH} generally fails when $N$ is assumed to be composite.   

\begin{theorem}
	\label{EvenN}Let $N$ be a composite number and let $H$ be a subgroup of the
	automorphism group of $\mathbb{Z}_{N}$ of order at least $p$, where $p$ is
	the smallest prime divisor of $N$. Then, $\mathbb{Z}_N\rtimes H$ is irreducibly
	spark deficient.
	
	
	If on the other hand, $N=p^n$ where $p$ a prime, $\abs{H}<p$ and no two elements of $H$ are congruent $\bmod p$, then
	there is an irreducible unitary representation that 
	is full spark. 
\end{theorem}

As we shall see in the proof, there are irreducible unitary representations of $\Z_N\rtimes H$ of maximal dimension $\abs{H}$
that are full spark, when $N=p^n$, $\abs{H}<p$, and no two elements of $H$ are congruent $\bmod p$. However, the group
$\Z_N\rtimes H$ is neither irreducibly full spark, nor spark deficient in this case, as there are other irreducible
unitary representations that are spark deficient.

\subsection{A central sufficient condition for the full spark property}

Let $T$ denote the cyclic shift operator in ${\mathbb{C}}^{N}$, i.e.
\[
T(x_{0},\dotsc,x_{N-1})=(x_{N-1},x_{0},\dotsc,x_{N-2}),
\]
for every vector $(x_{0},\dotsc,x_{N-1})\in{\mathbb{C}}^{N}$; the powers of $T$ will also be called
\emph{cyclic permutations}. A map of the
form
\[
(x_{0},x_{1},\dotsc,x_{N-1})\longmapsto(a_{0}x_{0},\dotsc,a_{N-1}x_{N-1})
\]
is a diagonal operator, and is denoted by $\operatorname{diag} (a_{0}%
,\dotsc,a_{N-1}) $; its matrix representation is
\[%
\begin{pmatrix}
a_{0} & 0 & \ldots & 0\\
0 & a_{1} & \ldots & 0\\
\vdots & \vdots & \ddots & \vdots\\
0 & 0 & \ldots & a_{N-1}%
\end{pmatrix}
.
\]
We will usually index the rows of an $N\times M$ matrix with elements from the
cyclic group $\mathbb{Z} _{N}$. In this setting, we will say that a submatrix
consists of \emph{consecutive} rows, if the corresponding elements of
$\mathbb{Z} _{N}$ are consecutive (more precisely, they form an arithmetic
progression with difference $1\bmod N$). For example, the submatrix consisting
of the top and bottom row of an $N\times M$ matrix consists of consecutive
rows, according to this definition.

Let $\Phi=\set{\v_1,\dotsc,\v_M}\ssq\C^N$ be a finite frame with frame bounds $B\geq A>0.$ 
	We will usually identify $\Phi$ with the $N\times M$ matrix whose columns are the frame elements, i.e.
	\[\Phi=\left(
	\begin{array}{cccc}
	\vertbar & \vertbar &        & \vertbar \\
	\v_1    & \v_2    & \ldots & \v_M    \\
	\vertbar & \vertbar &        & \vertbar 
	\end{array}
	\right).\]
The transpose of such a matrix is the matrix representation of the analysis operator. 

The following result is central to the proofs of our main results. 

\begin{theorem}
\label{maintool} \label{first} Suppose that the $N\times M$ matrix
\[
\mathcal{M}=%
\begin{pmatrix}
m_{00} & m_{01} & \ldots & m_{0,M-1}\\
m_{10} & m_{11} & \ldots & m_{1,M-1}\\
\vdots & \vdots & \ddots & \vdots\\
m_{N-1,0} & m_{N-1,1} & \ldots & m_{N-1,M-1}%
\end{pmatrix}
\]
has the following property: every minor consisting of consecutive rows from
$\mathcal{M}$ is nonzero. Then, there is an open dense subset $\mathcal{S}%
\subseteq\mathbb{C}^{N}$ such that for every $f\in\mathcal{S},$ the set
\[
\mathcal{M}_{\ell}T^{k}f,0\leq\ell\leq M-1,0\leq k\leq N-1,
\]
has the full spark property, where $\mathcal{M}_{\ell}=\operatorname{diag}%
(m_{0\ell},m_{1\ell},\dotsc,m_{N-1,\ell})$.
\end{theorem}

\subsubsection{Proof of Theorem \ref{first}}

Let $f=(z_{0},z_{1},\dotsc,z_{N-1})$. The determinant of the $N\times N$
formed by $N$ column vectors of the form (in a fixed order) $\mathcal{M}%
_{\ell}T^{k}f$ is a homogeneous polynomial in the coordinates of $f$ of degree
$N$. We will show that every such selection of $N$ vectors yields a nonzero
polynomial; the conclusion then follows easily, as we can take $\mathcal{S}$
to be the subset of $\mathbb{C}^{N}$  
where none of these (finitely many) 
polynomials vanish. The complement of $\mathcal{S}$ is a finite union of
varieties (closed subsets), having Lebesgue measure zero.

For this purpose, let 
\[\set{(k_{0},\ell_{0}),\dotsc,(k_{N-1},\ell_{N-1})}\ssq\mathbb{Z}_{N}\times{\left\{
{0,1,\dotsc,M-1}\right\}  }\] 
and define
\[
\Lambda=\set{\mc{M}_{\ell_j}T^{k_j}:0\leq j\leq N-1}
\]
and
\[
D_{\Lambda}(f)=%
\begin{pmatrix}
\mathcal{M}_{\ell_{0}}T^{k_{0}}f & \mathcal{M}_{\ell_{1}}T^{k_{1}}f & \ldots &
\mathcal{M}_{\ell_{N-1}}T^{k_{N-1}}f
\end{pmatrix}
,
\]
where all the $\mathcal{M}_{\ell_{j}}T^{k_{j}}f$ are column vectors. We
rewrite this matrix as
\[
D_{\Lambda}(f)=(D_{\La_0}(f)|D_{\La_1}(f)|\dotsb|D_{\La_{N-1}}(f)),
\]
where $D_{\La_j}(f)$ is the subset of columns of $D_{\Lambda}(f)$ for which $k_{i}=j$
holds, possibly by rearranging the columns of $D_{\Lambda}(f)$ (an operation that
only changes the sign of $\det(D_{\Lambda}(f))$). The determinant of $D_{\Lambda
}(f)$ is up to a sign equal to a sum of products of minors of the $D_{\La_j}(f)$
matrices. Suppose that $D_{\La_j}(f)$ is a $N\times L_{j}$ matrix (i.e. it consists
of $L_{j}$ columns). We also label the rows of $D_{\Lambda}(f)$ by ${\left\{
{0,1,\dotsc,N-1}\right\}  }$, with $0$ being the top row and $N-1$ being the
bottom one. Next, we consider a partition $(A_{0},A_{1},\dotsc,A_{N-1})$ of
${\left\{  {0,1,\dotsc,N-1}\right\}}$ such that $\abs{A_j}=L_{j}$, and
denote by $D_{\La_j}(A_{j})(f)$ the $L_{j}\times L_{j}$ minor of $D_{\La_j}(f)$ consisting
of the rows labeled by $A_{j}$ (in increasing order). As noted above,
$\det(D_{\Lambda}(f))$ is then a sum of products of the form
\begin{equation}
\pm\prod_{j=0}^{N-1}\det(D_{\La_j}(A_{j})(f)).\label{proddet}%
\end{equation}

When we expand the determinants in \eqref{proddet}, we obtain a monomial of
degree $N$ in $z_{0},\dotsc,z_{N-1}$. The coefficient of this monomial in
\eqref{proddet} is
\begin{equation}
\pm\prod_{j=0}^{N-1}\det(\mathcal{M}(A_{j},B_{j})),\label{Msub}%
\end{equation}
where $\mathcal{M}(A_{j},B_{j})$ is the submatrix of $\mc{M}$ with rows from
$A_{j}$ and columns from $B_{j}\subseteq{\left\{  {0,1,\dotsc,M-1}\right\}  }%
$. An element $\ell$ belongs to $B_{j}$ if and only if $\mathcal{M}_{\ell
}T^{j}f$ is a column of $D_{\Lambda}(f)$, which is equivalent to $\mc{M}_{\ell}T^j\in\Lambda$.

It is evident that if we use consecutive rows from the matrix $\mathcal{M}$,
the expression in \eqref{Msub} will be nonzero, by hypothesis. In particular,
we will use
\begin{align*}
P_{0} &  ={\left\{  {0,1,\dotsc,L_{0}-1}\right\}  }\\
P_{1} &  ={\left\{  {L_{0},L_{0}+1,\dotsc,L_{0}+L_{1}-1}\right\}  }\\
&  \vdots\\
P_{N-1} &  ={\left\{  {L_{0}+\dotsb+L_{N-2},\dotsc,N-1}\right\}  }.
\end{align*}
Thus, the monomial that corresponds to \eqref{proddet} has a nonzero
contribution to its coefficient in $\det(D_{\Lambda}(f))$ for this particular
partition of ${\left\{  {0,1,\dotsc,N-1}\right\}  }$; this monomial will
henceforth be called the \emph{consecutive index monomial}, or CI for short.
However, $\det(D_{\Lambda}(f))$ is the sum of terms \eqref{proddet} where the
$A_{j}$ run through the partitions of ${\left\{  {0,1,\dotsc,N-1}\right\}  }$
with $\abs{A_j}=L_{j}$; in order to find the coefficient of the CI monomial in
$\det(D_{\Lambda}(f))$, we have to identify all such partitions which yield this
monomial in \eqref{proddet}, and then sum all such contributions of the form \eqref{Msub}.

As we shall see, the CI monomial appears \emph{uniquely}, i.e. there is a
unique partition $A_{0},\dotsc,A_{N-1}$ that gives the CI monomial; by
hypothesis, the total contribution, that is, the coefficient in
$\det(D_{\Lambda}(f))$ must be nonzero, thus $\det(D_{\Lambda}(f))$ is eventually a
nonzero polynomial, as desired.


Permuting the coordinates of $f$ does not change the fact that
$\det(D_{\Lambda}(f))$ is a nonzero polynomial in the coordinates of $f$, namely $z_{0},\dotsc,z_{N-1}$. In
other words, if
\begin{equation}
\Lambda^{\prime}={\left\{  {\mathcal{M}_{\ell}T^{k-s}:(k,\ell)\in\Lambda
}\right\}  },\label{shift}%
\end{equation}
and $D_{\L^{\prime}}(f)$ is defined similarly, that is,
\[
D_{\Lambda^{\prime}}(f)=%
\begin{pmatrix}
\mathcal{M}_{\ell_{0}}T^{k_{0}-s}f & \mathcal{M}_{\ell_{1}}T^{k_{1}-s}f & \ldots &
\mathcal{M}_{\ell_{N-1}}T^{k_{N-1}-s}f
\end{pmatrix}
,
\]
then $\det(D_{\Lambda}(f))$ is a nonzero polynomial in $z_{0},\dotsc,z_{N-1}$ if
and only if $\det(D_{\Lambda^{\prime}}(f))$ is, where $s\in\mathbb{Z}_{N}$; indeed, since
\[\det(D_{\La^{\prime}}(f))=\det(D_{\La}(T^{-s}f)).\]
A
cyclic permutation of the coordinates results in a cyclic permutation of the
$P_{j}$, i.e. $P_{j}^{\prime}=P_{j+s}$, where the indices are taken $\bmod N$.
We define $L_{j}^{\prime}=\abs{P'_j}$, and claim that there is such a cyclic
permutation, for which

\begin{enumerate}[{\bf(I)}]
\item $L^{\prime}_{0}>0$. \label{1}
\item\label{2} The sets $P^{\prime}_{j}-j$ consist of non-negative integers; here, we
assume that all the $P^{\prime}_{j}$ are subsets of ${\left\{  {0,1,\dotsc
,N-1}\right\}  }$.
\end{enumerate}

This follows directly from the following combinatorial lemma by Dvoretzky and Motzkin, the so-called \emph{Cycle Lemma}
\cite{DM47} (this version here is actually the second part of Theorem 3 in \cite{DM47}):

\begin{lemma}
	Let $p_1p_2\dotsb p_{m+n}$ be a sequence of $m$ values equal to $+1$ and $n$ values equal to $-1$, such that $m\geq n$. Then, there are at least $m-n+1$ cyclic
	permutations of the form $q_1q_2\dotsb q_{m+n}$, such that all partial sums
	\[q_1+\dotsb+q_j\]
	are non-negative, for $1\leq j\leq m+n$.
\end{lemma}

We define $K_{j}=L_{0}+\dotsb+L_{j-1}$ ($K_{0}=0$), so that
\[
P_{j}={\left\{  {K_{j},K_{j}+1,\dotsc,K_{j+1}-1}\right\}  }.
\]
The set $P_{j}-j$ is then
\[
{\left\{  {K_{j}-j,K_{j}-j+1,\dotsc,K_{j+1}-j-1}\right\}  }%
\]
The second condition above is thus equivalent to
\begin{equation}\label{partialsums}
L_0+\dotsb+L_{j-1}\geq j.
\end{equation}
Next, we construct the following sequence of $\pm1$ based on the data given by $L_0,\dotsc,L_{N-1}$, as follows. We begin with a
string of $L_0$ $+1$'s, followed by a single $-1$, then a string of $L_1$ $+1$'s, followed by a single $-1$, $\dotsc$,
followed by a string of $L_{N-1}$ $+1$'s, and finally ending with a $-1$. Since $L_0+\dotsb+L_{N-1}=N$, the number of $+1$
is equal to the number of $-1$, therefore by the Cycle Lemma, there is a cyclic permutation, whose partial sums are all
non-negative. We can always choose such a permutation to end with a $-1$, i.e. it begins with a string of $L_s$ $+1$'s, for
some $s$. We then put $L'_j=L'_{j+s}$ (here the indices are taken $\bmod N$), so that $L'_0=L'_s$, and the conclusion of
the Cycle Lemma then is equivalent to the second condition, that is
\[L'_0+\dotsb+L'_{j-1}\geq j\]
for all $j$.



Without loss of generality, we will assume henceforth that conditions \eqref{1} and \eqref{2}
are satisfied for $\Lambda$ itself. We recall that the CI monomial corresponds to the standard $P_j$ partition. 

We note that any desired partition has the form $\sigma(P_{0}%
),\dotsc,\sigma(P_{N-1})$, where $P_{0},\dotsc,P_{N-1}$ is the above fixed
partition, which gives the CI monomial, and $\sigma$ runs through a quotient
subgroup of the permutation group $S_{N}$, namely $S_{N}/\Gamma$, where
$\Gamma$ the subset of permutations that leave each $P_{j}$ invariant (i.e.
$\sigma(P_{j})=P_{j}$ for all $j$); we will call such permutations
\emph{trivial}.

Therefore, every partition of ${\left\{  {0,1,\dotsc,N-1}\right\}  }$ into
$A_{j}$ with $\abs{A_j}=L_{j}$ is obtained through a permutation $\sigma\in
S_{N}/\Gamma$, by setting $A_{j}=\sigma(P_{j})$. With every such partition (or
equivalently, element of $S_{N}/\Gamma$) we associate a discrete random
variable $X_{\sigma}$ with values at ${\left\{  {0,1,\dotsc,N-1}\right\}  }$
satisfying
\[
\mathbb{P} [X_{\sigma}=j]=\frac{a_{j}}{N},
\]
where $Z^{\sigma}=z_{0}^{a_{0}}z_{1}^{a_{1}}\dotsb z_{N-1}^{a_{N-1}}$ is the
monomial associated with the partition $A_{0},\dotsc,A_{N-1}$. We put
$X=X_{\iota}$, where $\iota$ is the identity partition; in other words, $X$ is
the random variable associated with the CI monomial, which will be denoted by
$Z$. We will show that the CI monomial appears uniquely by showing that the
identity is the unique minimizer of the quantity $\E[X_{\sigma}^{2}]$ in
$S_{N}/\Gamma$\footnote{We note that the remaining of the proof follows in almost identical
fashion as the proof of Theorem 4.3 from \cite{M15}; we include the full proof for completeness.}.

To do so, define the function $g$ on $\set{0,1,\dotsc,N-1}$ with the property
\[g^{-1}(j)=P_j,\	\;\;\; 0\leq j\leq N-1.\]
Obviously, $g$ is increasing, and is constant when restricted to each of the $P_j$. Conditions \eqref{1} and \eqref{2}
above imply that $n-g(n)\geq0$, for all $n\in\set{0,\dotsc,N-1}$. We note that the indices appearing at the CI monomial
are exactly the integers $n-g(n)$, counting multiplicities. On the other hand, the indices appearing in the monomial
$Z^{\si}$ form precisely the union of \emph{multisets} (i.e. we count an integer multiple times, if it appears on many
sets) $\si(P_j)-j$, where the latter are taken $\bmod N$. So, for $0\leq n\leq N-1$, the corresponding index for $Z^{\si}$
is $\si(n)-g(n)$. Since it is less than $N$ in absolute value, this index \emph{as an integer} in $\set{0,\dotsc,N-1}$ must either be
$\si(n)-g(n)$ or $\si(n)-g(n)+N$. We thus define
\begin{equation}
\si'(n)=\begin{cases}
\si(n),\	\;\;\; & \text{ if } \si(n)-g(n)\geq0\\
\si(n)+N,\	\;\; & \text{ otherwise}.
\end{cases}
\end{equation}
Let $C_1$ be the set of $n$ for which $\si(n)-g(n)\geq0$ holds, and $C_2$ its complement in $\set{0,\dotsc,N-1}$. Define
\[f:\set{0,\dotsc,N-1}\longrightarrow \si(C_1)\cup(\si(C_2)+N)\]
to be strictly increasing ($f$ is obviously unique), so that $f(n)\geq n$, for all $n$. Since $ \si(C_1)\cup(\si(C_2)+N)$
is also the range of $\si'$, there must be a permutation $\tau$ of this set, such that
\[\si'(n)=\tau(f(n)),\	\;\;\; 0\leq n\leq N-1.\]
Now, since $f(n)\geq n$, we must have $f(n)-g(n)\geq n-g(n)\geq0$ for $0\leq n\leq N-1$, hence
\[\E[X^2]\leq\frac{1}{N}\sum_{n=0}^{N-1}(f(n)-g(n))^2.\]
On the other hand,
\begin{eqnarray*}
\E[X_{\si}^2]-\frac{1}{N}\sum_{n=0}^{N-1}(f(n)-g(n))^2 &=& \frac{1}{N}\sum_{n=0}^{N-1}(\si'(n)-g(n))^2-\frac{1}{N}\sum_{n=0}^{N-1}(f(n)-g(n))^2\\
&=& \frac{1}{N}\sum_{n=0}^{N-1}(\tau(f(n))-g(n))^2-\frac{1}{N}\sum_{n=0}^{N-1}(f(n)-g(n))^2\\
&=& \frac{2}{N}\sum_{n=0}^{N-1}f(n)g(n)-\frac{2}{N}\sum_{n=0}^{N-1}\tau(f(n))g(n)\\
&\geq& 0,
\end{eqnarray*}
where the latter inequality is justified by the \emph{rearrangement inequality}. Thus,
\[\E[X^2]\leq\frac{1}{N}\sum_{n=0}^{N-1}(f(n)-g(n))^2\leq\E[X_{\si}^2],\]
for all permutations $\si$. If $C_2\neq\varnothing$, then for some $n,$ we have $f(n)-g(n)>n-g(n)$, yielding strict
inequality between $\E[X^2]$ and $\E[X_{\si}^2]$. So, $\E[X^2]$ can only equal $\E[X_{\si}^2]$ when $C_2=\varnothing$, which
implies that $f$ must be the identity on $\set{0,\dotsc,N-1}$. It also implies that $\si'(n)=\si(n)$ for all $n$, which yields
\begin{eqnarray*}
	\E[X_{\si}^2]-\E[X^2] &=& \frac{1}{N}\sum_{n=0}^{N-1}(\si(n)-g(n))^2-\frac{1}{N}\sum_{n=0}^{N-1}(n-g(n))^2\\
	&=& \frac{2}{N}\sum_{n=0}^{N-1}ng(n)-\frac{2}{N}\sum_{n=0}^{N-1}\si(n)g(n)\\
	&\geq& 0,
\end{eqnarray*}
again by the rearrangement inequality; we note that equality can only occur when $\si$ is trivial, i.e. it leaves every
$P_j$ (the level sets of $g$) invariant. This clearly shows that the expected value of $X^2$ is a unique statistical property
for the CI monomial with respect to the given partition. Therefore, the CI monomial appears uniquely, and has a
nonzero coefficient by hypothesis in the expansion of $\det(D_{\L}(f))$, thus rendering the latter a nonzero polynomial in
$z_0,\dotsc,z_{N-1}$. The result now follows by observing that $\mc{M}_{\ell}T^kf$, $0\leq \ell\leq M-1$, $0\leq k\leq N-1$
produces a full spark frame as long as $f$ avoids the finitely many varieties defined by the zero sets of the polynomials
$\det(D_{\L}(f))$. \hspace{2.4cm}
$\square$

\subsubsection{A sufficient condition for tightness}

Next, we will examine when the frame $$\set{\mc{M}_{\ell}T^kf:0\leq\ell\leq M-1, 0\leq k\leq N-1}$$ is tight. 

\begin{theorem}
 Let $\mc{M}=\bra{m_{k\ell}}$ be a $N\times M$ matrix, and denote 
 \[\mc{M}_{\ell}=\mathrm{diag}(m_{0\ell},m_{1\ell},\dotsc,m_{N-1,\ell}).\] 
 Then, for a nonzero $f\in\C^N$, the frame $\set{\mc{M}_{\ell}T^kf:0\leq\ell\leq M-1, 0\leq k\leq N-1}$ is tight if the columns of $\mc{M}$ form a tight frame.
\end{theorem}

\begin{proof}
 It is known that the columns of a matrix $\mc{M}$ form a tight frame if and only if $\mc{M}\mc{M}^*$ is a multiple of the identity. If $\mc{M}^k$ denotes the $k$th row of $\mc{M}$,
 this is equivalent to the fact that the rows of $\mc{M}$ are orthogonal and have the same norm. Without loss of generality, we may assume $\norm{\mc{M}^k}=1$ for all $k$.
 Let $\Phi$ be the $N\times MN$ matrix whose columns are $\mc{M}_{\ell}T^kf$. Denote by $\Phi^k$ the $k$th row of $\Phi$. We have
 \[\norm{\Phi^k}^2=\norm{\mc{M}^k}^2\norm{f}^2=\norm{f}^2\]
 and
 \[\scal{\Phi^j,\Phi^k}=\scal{\mc{M}^j,\mc{M}^k}\sum_{m\in\Z_N}\scal{f,T^mf}=0,\]
 as desired.
\end{proof}

\subsection{A short overview of the theory of Mackey}\label{Mackey}

In order to present the proof of Theorem \ref{ZnH}, we need to describe (up to unitary equivalence) the set of all irreducible representations of the semi-direct product group $\mathbb{Z}_N\rtimes H.$ To this end, we will make use of Mackey's analysis \cite[Chapter 6]{folland2016course}.

Let $G=AH$ be a finite group where $A,H$ are subgroups, $A$ is normal in $G$,
abelian and $A\cap H=\left\{  e\right\}  $. In other words, $G$ is isomorphic
to the semi-direct product group $A\rtimes H.$ Since $A$ is abelian, its
unitary irreducible representations are all homomorphisms from $A$
into the torus group.

Let $\widehat{A}$ be the unitary dual of $A.$ That is, the collection of all
non-equivalent irreducible representations of $A.$ The conjugation action of
$G$ on $A$ induces an action on $\widehat{A}$ defined as follows. For $g\in
G,\chi\in\widehat{A}$ and $x\in A,$%
\[
\left[  g\cdot\chi\right]  \left(  x\right)  =\chi\left(  g^{-1}xg\right)  .
\]
Let $\Sigma\subseteq\widehat{A}$ be a transversal for the $G$-orbits in
$\widehat{A}.$ For each $\chi\in\Sigma,$ let $H_{\chi}$ be the stabilizer of
$\chi$ in $H.$ That is,
\[
H_{\chi}=\left\{  h\in H:h\cdot\chi=\chi\right\}  .
\]
Put
\[
G_{\chi}=AH_{\chi}.
\]
The map $gG_{\chi}\mapsto g\cdot\chi$ defines a bijection between $G/G_{\chi}
$ and the orbit of $\chi.$ Moreover, $G_{\chi}$ is the stabilizer of $\chi$ in
$G,$ and each unitary character $\chi$ can be extended to a unitary
representation $\omega_{\chi}$ of $G_{\chi}$ as follows. Given $n\in A$ and
$h\in H_{\chi},$%
\[
\omega_{\chi}\left(  nh\right)  =\chi\left(  n\right)  .
\]
Indeed, for $n,m\in A$ and $h,k\in H_{\chi},$ we verify that $\omega_{\chi
}\left(  \left(  nh\right)  \left(  mk\right)  \right)  =\omega_{\chi}\left(
nh\right)  \omega_{\chi}\left(  mk\right)  $ as follows
\begin{align*}
\omega_{\chi}\left(  nhmk\right)   &  =\omega_{\chi}\left(  \left(
nhmh^{-1}\right)  hk\right) \\
&  =\chi\left(  nhmh^{-1}\right) \\
&  =\chi\left(  n\right)  \chi\left(  hmh^{-1}\right) \\
&  =\chi\left(  n\right)  \left[  h^{-1}\cdot\chi\right]  \left(  m\right)  .
\end{align*}
Since $h^{-1}\cdot\chi=\chi,$%
\[
\omega_{\chi}\left(  nhmk\right)  =\chi\left(  n\right)  \chi\left(  m\right)
=\omega_{\chi}\left(  nh\right)  \omega_{\chi}\left(  mk\right)  .
\]
Next, given a unitary irreducible representation $\rho$ of $H_{\chi}$ acting
in $\mathcal{H}_{\rho},$ the tensor representation $\omega_{\chi}\otimes\rho$
is an irreducible representation of the subgroup $G_{\chi}.$ Letting
\[
\pi_{\rho,\chi}=\mathrm{ind}_{G_{\chi}}^{G}\left(  \omega_{\chi}\otimes
\rho\right)
\]
be the representation obtained by inducing $\omega_{\chi}\otimes\rho$ from
$G_{\chi}$ to $G,$ according to the theory of Mackey, $\pi_{\rho,\chi}$ is
necessarily irreducible \cite[Chapter 6]{folland2016course}. In fact, all
irreducible representations of $G$ arise in this fashion. Precisely, for each
$\chi\in\Sigma,$ letting $\widehat{H_{\chi}}$ be a collection consisting of
non-equivalent unitary irreducible representations of $H_{\chi},$ every
unitary irreducible representation of $G=AH$ is equivalent to one element of
the following set%
\[%
{\displaystyle\bigcup\limits_{\chi\in\Sigma}}
\left\{  \pi_{\rho,\chi}:\rho\in\widehat{H_{\chi}}\right\}  .
\]
Moreover, each $\pi_{\rho,\chi}$ is realized as acting in the space
$\mathcal{F}_{\rho,\chi}$ consisting of functions $f:G\rightarrow
\mathcal{H}_{\rho}$ satisfying the covariance relation
\begin{equation}
f\left(  xz\right)  =\left(  \omega_{\chi}\otimes\rho\right)  \left(
z^{-1}\right)  f\left(  x\right)  \text{ where }z\in G_{\chi}.
\label{covariance}%
\end{equation}
The vector space $\mathcal{F}_{\rho,\chi}$ is naturally endowed with the inner
product
\begin{equation}
\left\langle v,w\right\rangle _{\mathcal{F}_{\rho,\chi}}=\sum_{gH\in
G/G_{\chi}}\left\langle v\left(  g\right)  ,w\left(  g\right)  \right\rangle
_{\mathcal{H}_{\rho}}. \label{innerproduct}%
\end{equation}
Note that the covariance relation described in (\ref{covariance})\ implies
that the inner product (\ref{innerproduct}) is indeed well defined. \newline

The action of $\pi_{\rho,\chi}$ is defined as follows. For $f\in
\mathcal{F}_{\rho,\chi}$ and $g,z\in G,$%
\[
\pi_{\rho,\chi}\left(  g\right)  f\left(  z\right)  =f\left(  g^{-1}z\right)
.
\]
\begin{remark} Note that in cases where the subgroup $H_{\chi}$ is the singleton containing
the identity in $H$, the extension representation $\omega_{\chi}$ is trivial
in the sense that $\omega_{\chi}=\chi.$ It is thus convenient to realize the
action of $\mathrm{ind}_{A}^{G}\left(  \chi\right)  $ as acting in
$l^{2}\left(  H\right)  $ as follows. Since $H\subseteq G$ is a transversal
for $G/A,$ given $f\in\mathcal{F}_{1,\chi},$ we have
\[
f\left(  ha\right)  =\chi\left(  a^{-1}\right)  f\left(  h\right)  \text{
\ \ \ \ \ }\left(  a,h\right)  \in A\times H.
\]
Therefore, $\mathcal{F}_{1,\chi}$ can be naturally identified with
$l^{2}\left(  H\right)  $, and the following is immediate. For $f\in
\mathcal{F}_{1,\chi},$ $a\in A,$
\begin{align*}
\left[  \pi_{1,\chi}\left(  a\right)  f\right]  \left(  h\right)   &  =\left[
\mathrm{ind}_{A}^{G}\left(  \chi\right)  \left(  a\right)  f\right]  \left(
h\right) \\
&  =f\left(  a^{-1}h\right) \\
&  =f\left(  hh^{-1}a^{-1}h\right) \\
&  =\chi\left(  h^{-1}ah\right)  f\left(  h\right)
\end{align*}
and for $k\in H,$%
\[
\left[  \pi_{1,\chi}\left(  k\right)  f\right]  \left(  h\right)  =\left[
\mathrm{ind}_{A}^{G}\left(  \chi\right)  \left(  k\right)  f\right]  \left(
h\right)  =f\left(  k^{-1}h\right)  .
\]

In summary, elements of the normal part of the group act diagonally, while the restriction of $ \pi_{1,\chi}$ to $H$ is unitarily equivalent to the left regular representation of $H.$ \end{remark}

\subsection{Proof of Theorem \ref{ZnH}}

Let $\pi$ be an irreducible and unitary representation of $\mathbb{Z}%
_{N}\rtimes H$ acting in a Hilbert space $\mathcal{H}_{\pi}.$ Since $\pi$ is
irreducible \cite[Theorem 3]{waldron2018group} it is known that for any
nonzero vector $f\in\mathcal{H}_{\pi},$ $\left\{  \pi\left(  g\right)  f:g\in
G\right\}  $ is a tight frame for $\mathcal{H}_{\pi}$.

To prove Theorem \ref{ZnH}, it is enough to show that there exists a nonzero
vector whose orbit is a full spark frame in $\mathcal{H}_{\pi}$. Since the
case where $\pi$ is a unitary character is trivially true, we shall not
address it, and we will instead focus on representations of dimensions
strictly larger than one.

Since $N$ is assumed to be prime, the automorphism group of $\mathbb{Z}%
_{N}=\left\{  0,1,\cdots,N-1\right\}  $ is the multiplicative modular group
$\mathbb{Z}_{N}^{\times}=\left\{  1,\cdots,N-1\right\}  $ which is isomorphic
to $\mathbb{Z}_{N-1}.$

Let $G=\mathbb{Z}_{N}\rtimes\mathbb{Z}_{N}^{\times}$ be the semi-direct product
group endowed with the following multiplication law. For $x,y\in\mathbb{Z}%
_{N}=\left\{  0,1,\cdots,N-1\right\}  $ and $a,b\in\mathbb{Z}_{N}^{\times
}=\left\{  1,\cdots,N-1\right\}  ,$
\begin{equation}
\left(  x,a\right)  \left(  y,b\right)  =\left(  x+ay,ab\right)
.\label{product}%
\end{equation}
For a unitary character,
\[
\mathbb{\chi}_{\xi}:x\mapsto\exp\left(  \frac{2\pi i\xi x}{N}\right)  \text{
\ \ \ \ }\left(  \xi\in\mathbb{Z}_{N}\right)
\]
of $\mathbb{Z}_{N},$ and $a\in\mathbb{Z}_{N}^{\times},$ it is easy to verify
that
\[
a\cdot\mathbb{\chi}_{\xi}=\mathbb{\chi}_{a^{-1}\xi}.
\]
Under the action of $G,$ the unitary dual of $\mathbb{Z}_{N}$ is partitioned
into two classes of orbits:%
\[
\left\{  \mathbb{\chi}_{0}\right\}  \overset{\cdot}{\cup}\left\{
\mathbb{\chi}_{1},\dotsc,\mathbb{\chi}_{N-1}\right\}  .
\]
$\Sigma=\left\{  \mathbb{\chi}_{0},\mathbb{\chi}_{1}\right\}  $ is a
transversal for the orbits of $G$ in the unitary dual of $\mathbb{Z}_{N}.$
Moreover, the stabilizer subgroup of $\mathbb{\chi}_{0}$ coincides with $G$
while the stabilizer subgroup corresponding to the representation $\chi
_{1}\mathbb{\ }$is trivial. Since $\mathbb{Z}_{N}^{\times}$ is abelian, it
follows from Mackey's analysis that up to unitary equivalence, the unitary
dual of $G$ is parametrized by the set%
\begin{equation}
\left\{  \omega_{\mathbb{\chi}_{0}}\otimes\rho:\rho\in\widehat{\mathbb{Z}%
_{N}^{\times}}\right\}  \cup\left\{  \pi=\mathrm{ind}_{\mathbb{Z}_{N}%
\rtimes\left\{  \left(  0,1\right)  \right\}  }^{G}\left(  \mathbb{\chi}%
_{1}\right)  \right\}  .
\end{equation}
On the one hand, all representations in $\left\{  \omega_{\mathbb{\chi}_{0}%
}\otimes\rho:\rho\in\widehat{\mathbb{Z}_{N}^{\times}}\right\}  $ are
characters, and there are $\left(  N-1\right)  $-many of them. On the other,
every irreducible representation of $G$ of dimension larger than one is
unitarily equivalent to $ \pi=\mathrm{ind}_{\mathbb{Z}_{N}%
\rtimes\left\{  \left(  0,1\right)  \right\}  }^{G}\left(  \mathbb{\chi}%
_{1}\right).$  The
representation $\pi$ is realized as acting in $l^{2}\left(  \mathbb{Z}%
_{N}^{\times}\right)  $ as follows. Letting $a$ be a generator of the
multiplicative group $\mathbb{Z}_{N}^{\times}$ and $\omega=\exp\left(
\frac{2\pi i}{N}\right)  ,$
\[
\pi\left(  \gamma\right)  f\left(  a^{j}\right)  =\left\{
\begin{array}
[c]{c}%
\omega^{a^{-j}x}f\left(  a^{j}\right)  \text{ if }\gamma=\left(  x,1\right)
\\
f\left(  a^{j-\kappa}\right)  \text{ if }\gamma=\left(  0,a^{\kappa}\right)
\end{array}
\right.
\]
for $f\in l^{2}\left(  \mathbb{Z}_{N}^{\times}\right)  =\mathbb{C}^{N-1}.$
Next, the matrix $\mathcal{M}$ defined in Theorem \ref{first} takes the form
\begin{equation}
\mathcal{M}=\mathcal{M}_{N}\left(  \omega\right)  =\left(
\begin{array}
[c]{ccccc}%
1 & \omega & \omega^{2} & \cdots & \omega^{\left(  N-1\right)  }\\
1 & \omega^{a^{-1}} & \omega^{2a^{-1}} & \cdots & \omega^{\left(  N-1\right)
a^{-1}}\\
1 & \omega^{a^{-2}} & \omega^{2a^{-2}} & \cdots & \omega^{\left(  N-1\right)
a^{-2}}\\
\vdots & \vdots & \vdots & \ddots & \vdots\\
1 & \omega^{a^{-\left(  N-2\right)  }} & \omega^{2a^{-\left(  N-2\right)  }} &
\cdots & \omega^{\left(  N-1\right)  a^{-\left(  N-2\right)  }}%
\end{array}
\right)  =\left(  \omega^{ka^{-j}}\right)  _{0\leq j\leq N-2,0\leq k\leq N-1}.
\end{equation}
Precisely, $\mathcal{M}$ is an $\left(  N-1\right)  \times N$ matrix obtained
by arranging the diagonals of $\pi\left(  x,1\right)  ,x\in\mathbb{Z}_{N}$ in columns.

Suppose on the other hand that $H$ is a proper subgroup of $\mathbb{Z}%
_{N}^{\times}.$ Clearly, the order of $H$ must strictly divide $N-1$. Let
$\Gamma=\mathbb{Z}_{N}\rtimes H$ be the semi-direct product group with a
multiplication law defined similarly to (\ref{product}). Note that since $N$
is a prime number, every nonzero element of $\mathbb{Z}_{N}$ is a non--zero
divisor and as such, the stabilizer subgroup of $\mathbb{\chi}_{\xi}$ for
$\xi\in\mathbb{Z}_{N}\backslash\left\{  0\right\}  $ is the trivial subgroup
of $H.$ In other words, if $\pi$ is an irreducible representation of $\Gamma$
which is not a character, according to the theory of Mackey, there exists
$\xi\in\mathbb{Z}_{N}\backslash\left\{  0\right\}  $, such that $\pi$ is
unitarily equivalent to a representation induced by the unitary character
\[
\mathbb{\chi}_{\xi}:\mathbb{Z}_{N}
\ni x\mapsto\exp\left(  \frac{2\pi
i\xi x}{N}\right)  \text{ \ \ \ \ \ }\left(  \xi\in\mathbb{Z}_{N}%
\backslash\left\{  0\right\}  \right)  .
\]
Such a representation may be modeled as acting in $l^{2}\left(  H\right)  $ as
follows:
\[
\pi_{\xi}\left(  x,s\right)  f\left(  h\right)  =\omega_{\xi}^{h^{-1}%
x}f\left(  s^{-1}h\right)
\]
for $\omega_{\xi}=\exp\left(  \frac{2\pi\xi i}{N}\right)  .$ \\It is perhaps
worth noting that if $\pi$ is an irreducible representation of $\mathbb{Z}%
_{N}\rtimes\mathbb{Z}_{N}^{\times}$ acting on $\mathcal{H}_{\pi}$ and if
additionally, $\pi$ is not a character, then
\[
\dim\left(  \mathcal{H}_{\pi}\right)  =N-1>\dim\left(  l^{2}\left(  H\right)
\right)  =\left\vert H\right\vert .
\]
This observation implies that the restriction of an irreducible representation
of $\mathbb{Z}_{N}\rtimes\mathbb{Z}_{N}^{\times}$ which is not a character to
its subgroup $\Gamma$ can never be equivalent to an irreducible representation
of $\Gamma.$ In any case, the matrix $\mathcal{M}$ as defined in Theorem
\ref{first} takes the form
\[
\left(  \omega_{\xi}^{ka^{-j}}\right)  _{0\leq j\leq\left\vert H\right\vert
-1,0\leq k\leq N-1}.
\]
In either case any square submatrix of $\mathcal{M}$ satisfies the sufficient
conditions described in the following result which was established in
\cite[Theorem 6]{evans1976generalized}.

\begin{proposition}
\label{IsaacEvans} Let $N$ be a prime number and let $\omega$ be a $N$-th root
of unity. Suppose that $a_{1},\cdots,a_{n}\in\mathbb{Z}$ are pairwise
incongruent modulo $N$. Assuming the same for $b_{1},\cdots,b_{n}\in
\mathbb{Z}$, the matrix $\left(  \omega^{a_{j}b_{k}}\right)  _{1\leq k,j\leq
n}$ is invertible.
\end{proposition}

Appealing to Theorem \ref{first}, we conclude that if $H$ is
any subgroup of the automorphism group of $\mathbb{Z}_{N},$ $N$ is a prime
natural number, and $\pi$ is a unitary irreducible representation of
$\mathbb{Z}_{N}\rtimes H$ acting in some Hilbert space $\mathbb{C}^{M}$, then
for almost every vector $f\in\mathbb{C}^{M}$ (with respect to the Lebesgue
measure on $\mathbb{C}^{M}$) the orbit of $f$ is an equal norm tight frame which is also
full spark.

\begin{example}
	Let $	G=\mathbb{Z}_{5}\rtimes\mathbb{Z}_{4}=\left\{  0,1,2,3,4\right\}
	\rtimes\left\{  1,2,3,4\right\}.$ 	Appealing to Mackey's analysis, the
	unitary dual of $G$ contains (up to equivalence) four unitary characters and one irreducible
	representation $\pi$ acting in $\mathbb{C}^{4}.$  $\pi=\mathrm{ind}_{\mathbb{Z}_{5}\rtimes\left\{
		1\right\}  }^{G}\left(  \chi_{1}\right)$ is realized as acting in
	$\mathbb{C}^{4}$ as follows:
	\[
	\pi\left(  0,2\right)  =T=\left(
	\begin{array}
	[c]{cccc}%
	0 & 1 & 0 & 0\\
	0 & 0 & 1 & 0\\
	0 & 0 & 0 & 1\\
	1 & 0 & 0 & 0
	\end{array}
	\right)  \text{ and }\pi\left(  1,1\right)  =D=\left(
	\begin{array}
	[c]{cccc}%
	e^{\frac{2\pi i}{5}} & 0 & 0 & 0\\
	0 & e^{\frac{2\pi i3}{5}} & 0 & 0\\
	0 & 0 & e^{\frac{2\pi i3^{2}}{5}} & 0\\
	0 & 0 & 0 & e^{\frac{2\pi i3^{3}}{5}}%
	\end{array}
	\right)  .
	\]
	According to Theorem \ref{ZnH}, since $5$ is a prime number, we know that
	$\pi$ is full spark. To construct a fullspark frame, we proceed as
	follows. Let $v\in\mathbb{C}^{4}$ such that $v_{k}=t^{k^2}$ for
	$k\in\left\{  0,\cdots,3\right\}.$ The matrix representation of the corresponding analysis operator is of the form $C=\left(
\begin{array}{c}
 A \\ \hline
 B
\end{array}
\right)$ where

$$A=\left(
\begin{array}{cccc}
 1 & t & t^4 & t^9 \\
 t & t^4 & t^9 & 1 \\
 t^4 & t^9 & 1 & t \\
 t^9 & 1 & t & t^4 \\
 \omega  & t \omega ^3 & t^4 \omega ^4 & t^9 \omega ^2 \\
 t \omega  & t^4 \omega ^3 & t^9 \omega ^4 & \omega ^2 \\
 t^4 \omega  & t^9 \omega ^3 & \omega ^4 & t \omega ^2 \\
 t^9 \omega  & \omega ^3 & t \omega ^4 & t^4 \omega ^2 \\
 \omega ^2 & t \omega ^6 & t^4 \omega ^8 & t^9 \omega ^4 \\
 t \omega ^2 & t^4 \omega ^6 & t^9 \omega ^8 & \omega ^4
\end{array}
\right),B=\left(
\begin{array}{cccc}
 t^4 \omega ^2 & t^9 \omega ^6 & \omega ^8 & t \omega ^4 \\
 t^9 \omega ^2 & \omega ^6 & t \omega ^8 & t^4 \omega ^4 \\
 \omega ^3 & t \omega ^9 & t^4 \omega ^{12} & t^9 \omega ^6 \\
 t \omega ^3 & t^4 \omega ^9 & t^9 \omega ^{12} & \omega ^6 \\
 t^4 \omega ^3 & t^9 \omega ^9 & \omega ^{12} & t \omega ^6 \\
 t^9 \omega ^3 & \omega ^9 & t \omega ^{12} & t^4 \omega ^6 \\
 \omega ^4 & t \omega ^{12} & t^4 \omega ^{16} & t^9 \omega ^8 \\
 t \omega ^4 & t^4 \omega ^{12} & t^9 \omega ^{16} & \omega ^8 \\
 t^4 \omega ^4 & t^9 \omega ^{12} & \omega ^{16} & t \omega ^8 \\
 t^9 \omega ^4 & \omega ^{12} & t \omega ^{16} & t^4 \omega ^8
\end{array}
\right)$$ and $\omega=e^{\frac{2 \pi i}{5}}.$ Letting $t$ be any transcendental complex number over the field $\Q(\omega ),$ the orbit of $v$ is a full spark and tight frame consisting of twenty vectors.
\end{example}

\subsection{Proof of Theorem \ref{EvenN}}

Let $N$ be a composite number and let $H$ be a subgroup of the automorphism
group of $\mathbb{Z}_{N}$ of order at least $p$, where $p$ is the smallest
prime factor of $N$. We first consider the unitary irreducible representations of
$\Z_N\rtimes H$ that are induced from characters of $\widehat{\mathbb{Z}_{N}}$ of
maximal order (see subsection \ref{Mackey}, especially the Remark towards the end).
Such a character has the form $\chi_m$ 
defined by
\[
\chi_m:x\mapsto\exp{\left(  {\frac{2\pi imx}{N}}\right)  }%
\]
where $\gcd(m,N)=1$, and
its stabilizer in $G=\mathbb{\ Z}_{N}\rtimes H$ is $G_{1}=\mathbb{\ Z}%
_{N}\rtimes{\left\{  {1}\right\}  }.$ Appealing to Mackey's analysis, the
representation
\[
\pi_m=\mathrm{ind}_{G_{1}}^{G}\chi_m%
\]
is an irreducible representation acting on a space $\mathcal{H}_{m}$ of
dimension $\abs{H}\geq p$. Actually, $\mathcal{H}_{m}$ is the direct sum of
one-dimensional spaces $W_{h}=(0,h)W_{1}$, indexed by $h\in H$. A generic
element $v$ of $\mathcal{H}_{m}$ is then written as
\[
v=\sum_{h\in H}(0,h)w_{h},
\]
where $w_{h}\in W_{1}=W$, where $W$ is the one-dimensional space on which
$\chi_m$ acts. Obviously, ${\left\{  {0}\right\}  }\rtimes H$ acts on
$\mathcal{H}_{m}$ by permuting coordinates (i.e. $\pi_m|H$ is just the regular
representation of $H$). We only need to determine the action of the other
factor, namely $\mathbb{\ Z}_{N}\rtimes{\left\{  {1}\right\}  }$. For arbitrary 
$x\in\mathbb{\ Z}_{N},$ we obtain
\[
(x,1)v=\sum_{h\in H}(x,h)w_{h}=\sum_{h\in H}(0,h)(h^{-1}x,1)w_{h}=\sum_{h\in
H}\exp{\left(  {\frac{2\pi imh^{-1}x}{N}}\right)  }(0,h)w_{h}.
\]
Hence $\mathbb{\ Z}_{N}\rtimes{\left\{  {1}\right\} }$ acts diagonally on
$\mathcal{H}_{m}$. The orbit of $v\in\mathcal{H}_{m}$ under the action of
$\mathbb{\ Z}_{N}\rtimes{\left\{  {1}\right\}  }$ is the set of vectors that
appear as columns of the $\mathcal{M}^{\top}v$, where $\mathcal{M}$ is the
$\abs{H}\times N$ submatrix of the Fourier matrix of $\mathbb{\ Z}_{N}$, whose
rows are indexed by $m\cdot H$. If the columns of $\mathcal{M}$ form a spark
deficient set, then $\mathcal{M}^{\top}v$ is also spark deficient for every
choice of $v$, proving Theorem \ref{EvenN}.\\

For this purpose, we invoke the following definition and theorem by Alexeev,
Cahill, and Mixon \cite{ACM12}.

\begin{defn}
\cite[Definition 7]{ACM12} We say that a subset $\mathcal{S}\subseteq
\mathbb{\ Z}_{N}$ is \emph{uniformly distributed over the divisors of }$N$ if,
for every divisor $d$ of $N$, the $d$ cosets of ${\left\langle {d}%
\right\rangle }$ partition $\mathcal{S}$ into subsets, each of size
${\lfloor{\frac{\abs{\mc{S}}}{d}}\rfloor}$ or ${\lceil{\frac{\abs{\mc{S}}}{d}%
}\rceil}$.
\end{defn}

\begin{theorem}
\label{ACM}\cite[Theorem 9]{ACM12} Select rows indexed by $\mathcal{S}%
\subseteq\mathbb{\ Z}_{N}$ from the $N\times N$ discrete Fourier transform
matrix to build the submatrix F. If $F$ is full spark, then $\mathcal{S}$ is
uniformly distributed over the divisors of $N$. If in addition, $N$ is a prime
power, then the converse holds as well: if $\mathcal{S}$ is uniformly
distributed over the divisors of N, then $F$ is full spark.
\end{theorem}

We conclude the proof of Theorem \ref{EvenN} simply by noting that if $p$ is
the smallest prime factor of $N$ and $H<\mathbb{\ Z}_{N}^{\times}$ is a
subgroup of order at least $p$, then $m\cdot H$ cannot be uniformly distributed over
the divisors of $N$ for $\gcd(m,N)=1$, as it cannot be uniformly distributed $\bmod p$, since no
element in $H$ can be a multiple of $p$, hence
\[
0=\abs{H\cap p\Z_N}<1\leq{\lfloor{\frac{\abs{H}}{p}}\rfloor},
\]
thus the columns of $\mathcal{M}$ form a spark deficient set, as desired.

Regarding the unitary irreducible representations of $\Z_N\rtimes H$
which arise from $\chi_{m}$, where $\gcd(m,N)=d>1$, we simply observe that 
$\chi_m$ is not faithful, therefore, all the corresponding
representations of $\mathbb{\ Z}_{N}\rtimes H$ that are obtained through
Mackey's method, are not faithful, hence spark deficient, proving the first part of the Theorem. 
We remark that if we mod out the
kernel, we obtain a faithful irreducible representation of $\mathbb{\ Z}%
_{N/d}\rtimes H$, that arises from the character
\[
\psi_{m}:\Z_{N/d}\ni x\mapsto\exp{\left(  {\frac{2\pi im x}{N}}\right)  }.
\]

For the second part of Theorem \ref{EvenN}, we consider again the characters $\chi_m$ of $\Z_N$ with
$\gcd(m,N)=1$ and we apply the reverse direction of Theorem \ref{ACM}
for prime powers, as any minor of the resulting matrix $\mc{M}$ will consist of rows indexed by a 
subset of $\Z_N$ which has fewer than $p$ elements, no two of which are congruent $\bmod p$. By definition, this is an
equidistributed set with respect to any divisor of $p^n$, therefore, \emph{any} minor of $\mc{M}$ is nonzero; 
the proof then follows immediately by Theorem \ref{maintool}, concluding the proof of Theorem \ref{EvenN}. \qedhere

\begin{remark}
	We notice that the full spark representation of $\Z_N\rtimes H$ mentioned in the proof has dimension $\abs{H}$,
	where $N=p^n$ a prime power and $\abs{H}<p$. Furthermore, the subgroup $H$ must have distinct elements $\bmod p$; this
	forces $H$ to be a subgroup of the unique subgroup of $\Z_N^{\times}$ of order dividing $p-1$.
	 Therefore, for any $d$ and a prime
	$p\equiv 1\bmod d$, there is a full spark equal norm tight frame in $\C^d$ consisting of $dp^n$ elements, for every positive integer
	$n$.
\end{remark}

\begin{corollary}
Let $N$ be even and $H$ be a subgroup of $\mathbb{\ Z}_{N}^{\times}$. Define
the irreducible representation $\pi$ of $G=\mathbb{\ Z}_{N}\rtimes H$ as
above, acting on $\mathcal{H} _{\pi}$. Then, the set ${\left\{  {\pi
(g)v}\right\}  }_{g\in G}$ is spark deficient for every $v\in\mathcal{H}
_{\pi}$. If $N$ is odd and $H={\left\{  {1,-1}\right\}  }$, then there exist vectors
$v\in\mathcal{H} _{\pi}$, such that ${\left\{  {\pi(g)v}\right\}  }_{g\in G}$
is full spark.
\end{corollary}

\begin{proof}
The first part is an immediate consequence of Theorem \ref{ACM}, while the
second part follows from Theorem \ref{maintool}.
\end{proof}

In the general case where $N$ is composite, the representation $\pi$ can only be full spark
if $\abs{H}<p$, where $p$ is the smallest prime factor of $N$, and $H$ does not have two elements congruent
$\bmod q$, for any prime factor $q$ of $N$. This, in particular, means that the natural $\bmod q$ reduction
$H\to\Z_q^{\times}$ is injective, showing that $H$ must be cyclic. This proves:

\begin{corollary}
 With notation as above, if $H<\Z_N^{\times}$ is non-cyclic, then $\pi$ is spark deficient.
\end{corollary}

\begin{remark} It is known that the automorphism group of $\Z_N$ is cyclic if and only if $N=1,2,4,p^k,2p^k$ where $p$ is some positive odd prime number. In light of the preceding result, it is then clear that for any other values of $N,$ any irreducible representation of $\Z_N\rtimes \Z_N^{\times}$ obtained by inducing a unitary character of the normal group $\Z_N$ is spark deficient. \end{remark}

\section{\texorpdfstring{A generalization of the case of $\mathbb{Z}_{N}\rtimes H$ without an underlying group structure}{}}\label{nogroup}

Our primary objective in this section is to generalize Theorem \ref{ZnH}. Let

\[
T=\left(
\begin{array}
[c]{ccccc}%
0 & 1 & 0 & \cdots & 0\\
0 & 0 & 1 & \cdots & \vdots\\
\vdots & \vdots & 0 & \ddots & 0\\
0 & \cdots & 0 & \ddots & 1\\
1 & 0 & 0 & \cdots & 0
\end{array}
\right)
\]
and fix $\left(  \xi_{0},\xi_{1},\cdots,\xi_{N-1}\right)  \in\mathbb{Z}^{N}$
such that
\[
0\leq\xi_{0}<\xi_{1}<\cdots<\xi_{N-1}.
\]
For a real number $\lambda,$ define
\[
D_{\lambda}=\left(
\begin{array}
[c]{cccc}%
e^{2\pi i\lambda\xi_{0}} & 0 & \cdots & 0\\
0 & e^{2\pi i\lambda\xi_{1}} & \cdots & 0\\
\vdots & \vdots & \ddots & \vdots\\
0 & 0 & \cdots & e^{2\pi i\lambda\xi_{N-1}}%
\end{array}
\right)  .
\]
Next, let $\left\{  D_{\lambda}T^{\kappa}:\lambda\in\Lambda_{\tau},\kappa
\in\mathbb{Z}_{N}\right\}  $ be a collection of matrices unitary matrices
where%
\[
\Lambda_{\tau}=\left\{  0\leq\tau\lambda_{0}<\tau\lambda_{1}<\cdots
<\tau\lambda_{M-1}\right\}
\]
is a finite subset of the real line, $\tau$ is a fixed real number and
$\lambda_{0},\cdots,\lambda_{M-1}$ are some positive integers. We note that
the collection of matrices $\left\{  D_{\lambda}T^{\kappa}:\lambda\in
\Lambda_{\tau},\kappa\in\mathbb{Z}_{N}\right\}  $ rarely has the structure of
a group. 

\begin{remark}
For the specific case where $\left(  \xi_{0},\xi_{1},\cdots,\xi_{N-1}\right)
=\left(  0,1,\cdots,N-1\right)  $ and $\lambda=\frac{1}{N},$ we obtain
\[
D_{\frac{1}{N}}=\left(
\begin{array}
[c]{cccc}%
e^{\frac{2\pi i\left(  0\right)  }{N}} & 0 & \cdots & 0\\
0 & e^{\frac{2\pi i}{N}} & \cdots & 0\\
\vdots & \vdots & \ddots & \vdots\\
0 & 0 & \cdots & e^{\frac{2\pi i\left(  N-1\right)  }{N}}%
\end{array}
\right)  .
\]
Under these assumptions, it is quite easy to verify that the group generated
by $D_{\frac{1}{N}},T$ is the discrete Heisenberg group. Moreover, the group
$\left\langle D_{\frac{1}{N}},T\right\rangle $ acts unitarily and
irreducibly on $\mathbb{C}^{N}$ and it is known that there exists a vector
$f\in\mathbb{C}^{N}$ such that $\left\{  D_{\frac{1}{N}}^{m}T^{k}f:\left(
m,k\right)  \in\mathbb{Z}_{N}\times\mathbb{Z}_{N}\right\}  $ is a fullspark
frame for $\mathbb{C}^{N}.$
\end{remark}

We shall establish the following result.

\begin{theorem}
\label{general} If $e^{2\pi i\tau}$ is either transcendental or algebraic over
the rationals with sufficiently large degree, then there exists a nonzero
vector $f\in\mathbb{C}^{N}$ such that the collection $\left\{  D_{\lambda}T^{\kappa}%
f:\lambda\in\Lambda_{\tau},\kappa\in\mathbb{Z}_{N}\right\}  $ is a full spark
frame for $\mathbb{C}^{N}.$
\end{theorem}

The matrix $\mathcal{M}$ as defined in Theorem \ref{first} takes the form
\[
\mathcal{M}:=\mathcal{M}_{\tau}=\left(
\begin{array}
[c]{cccc}%
e^{2\pi i\left(  \tau\lambda_{0}\right)  \xi_{0}} & e^{2\pi i\left(
\tau\lambda_{1}\right)  \xi_{0}} & \cdots & e^{2\pi i\left(  \tau\lambda
_{M-1}\right)  \xi_{0}}\\
e^{2\pi i\left(  \tau\lambda_{0}\right)  \xi_{1}} & e^{2\pi i\left(
\tau\lambda_{1}\right)  \xi_{1}} & \cdots & e^{2\pi i\left(  \tau\lambda
_{M-1}\right)  \xi_{1}}\\
\vdots & \vdots & \ddots & \vdots\\
e^{2\pi i\left(  \tau\lambda_{0}\right)  \xi_{N-1}} & e^{2\pi i\left(
\tau\lambda_{1}\right)  \xi_{N-1}} & \cdots & e^{2\pi i\left(  \tau
\lambda_{M-1}\right)  \xi_{N-1}}%
\end{array}
\right)  .
\]

\vskip 0.5cm 

To prove Theorem \ref{general}, we will need some additional tools.
 A \textit{Young diagram} is a collection of cells arranged in the form
of left-justified and weakly decreasing number of cells in each row. By
listing the number of cells in each row, we obtain a partition of a fixed
positive integer which represents the total number of cells in the diagram. A
\textit{partition} $\kappa=\left(  \kappa_{1},\cdots,\kappa_{m}\right)  $ is a
finite sequence of weakly decreasing positive integers, and a \textit{Young
tableau} is a filling that is (a) weakly increasing across each row and (b)
strictly increasing down each column. A Young tableau obtained by filling a
Young diagram associated with a partition $\kappa$ is called a tableau of
shape $\kappa.$

 For each partition $\kappa$, there is a corresponding symmetric
polynomial known as a \textit{Schur polynomial} described as follows. Given
any filling $\mu$ of a Young diagram, we define a monomial $x^{\mu}$ which is
obtained by taking the product of the variables $x_{k}$ corresponding to the
indices $k$ appearing in $\mu.$ More precisely, such a monomial takes the
form
\[
x^{\mu}=%
{\displaystyle\prod\limits_{k=1}^{m}}
x_{k}^{\text{number of times }k\text{ occurs in }\mu}.
\]
The Schur polynomial
\[
s_{\kappa}\left(  x_{1},x_{2},\cdots,x_{m}\right)  =\sum x^{\mu}%
\]
is obtained by adding of all such monomials coming from tableaux $\mu$ of
shape $\kappa$ using the natural numbers $1,\cdots,m.$

\subsection{Proof of Theorem \ref{general}}

In light of Theorem \ref{maintool}, to establish the first part of the result, it is enough to show that if the
complex number $e^{2\pi i\tau}$ is transcendental or
algebraic over the rationals with a sufficiently large degree of extension,
then every minor of $\mathcal{M}$ is nonzero. Let $S$ be a submatrix of
$\mathcal{M}$ of order $L.$ There exists an elementary (permutation) matrix
$E$ of order $n$ such that
\[
ES=\left(  e^{2\pi i\tau\lambda_{s_{k}}\xi_{r_{j}}}\right)  _{0\leq j\leq
L-1,0\leq k\leq L-1}=\left(
\begin{array}
[c]{cccc}%
e^{2\pi i\tau\lambda_{s_{0}}\xi_{r_{0}}} & e^{2\pi i\tau\lambda_{s_{1}}%
\xi_{r_{0}}} & \cdots & e^{2\pi i\tau\lambda_{s_{L-1}}\xi_{r_{0}}}\\
e^{2\pi i\tau\lambda_{s_{0}}\xi_{r_{1}}} & e^{2\pi i\tau\lambda_{s_{1}}%
\xi_{r_{1}}} & \cdots & e^{2\pi i\tau\lambda_{s_{L-1}}\xi_{r_{1}}}\\
\vdots & \vdots & \ddots & \vdots\\
e^{2\pi i\tau\lambda_{s_{0}}\xi_{r_{L-1}}} & e^{2\pi i\tau\lambda_{s_{1}}%
\xi_{1}} & \cdots & e^{2\pi i\tau\lambda_{s_{L-1}}\xi_{r_{L-1}}}%
\end{array}
\right)
\]
for some finite sequences $\left(  r_{j}\right)  _{j=0}^{L-1}$ and $\left(
s_{j}\right)  _{j=0}^{L-1}$ where $0\leq\xi_{r_{0}}<\xi_{r_{1}}<\cdots
<\xi_{r_{L-1}}.$ Put
\[
P_{S}\left(  t\right)  =\left\vert
\begin{array}
[c]{cccc}%
t^{\lambda_{s_{0}}\xi_{r_{0}}} & t^{\lambda_{s_{1}}\xi_{r_{0}}} & \cdots &
t^{\lambda_{s_{L-1}}\xi_{r_{0}}}\\
t^{\lambda_{s_{0}}\xi_{r_{1}}} & t^{\lambda_{s_{1}}\xi_{r_{1}}} & \cdots &
t^{\lambda_{s_{L-1}}\xi_{r_{1}}}\\
\vdots & \vdots & \ddots & \vdots\\
t^{\lambda_{s_{0}}\xi_{r_{L-1}}} & t^{\lambda_{s_{1}}\xi_{r_{L-1}}} & \cdots &
t^{\lambda_{s_{L-1}}\xi_{r_{L-1}}}%
\end{array}
\right\vert .
\]
Appealing to the multi-linearity of the determinant function, we obtain
\[
P_{S}\left(  t\right)  =\left(
{\displaystyle\prod\limits_{j=0}^{L-1}}
t^{\lambda_{s_{0}}\xi_{r_{j}}}\right)  \cdot\det\left(  t^{\left(
\lambda_{s_{i}}-\lambda_{s_{0}}\right)  \xi_{r_{j}}}\right)  _{0\leq i\leq
L-1,0\leq j\leq L-1}%
\]
where
\[
\det\left(  t^{\left(  \lambda_{s_{i}}-\lambda_{s_{0}}\right)  \xi_{r_{j}}%
}\right)  _{0\leq i\leq L-1,0\leq j\leq L-1}=\left\vert
\begin{array}
[c]{cccc}%
1 & t^{\left(  \lambda_{s_{1}}-\lambda_{s_{0}}\right)  \xi_{r_{0}}} & \cdots &
t^{\left(  \lambda_{s_{L-1}}-\lambda_{s_{0}}\right)  \xi_{r_{0}}}\\
1 & t^{\left(  \lambda_{s_{1}}-\lambda_{s_{0}}\right)  \xi_{r_{1}}} & \cdots &
t^{\left(  \lambda_{s_{L-1}}-\lambda_{s_{0}}\right)  \xi_{r_{1}}}\\
\vdots & \vdots & \ddots & \vdots\\
1 & t^{\left(  \lambda_{s_{1}}-\lambda_{s_{0}}\right)  \xi_{r_{L-1}}} & \cdots
& t^{\left(  \lambda_{s_{L-1}}-\lambda_{s_{0}}\right)  \xi_{r_{L-1}}}%
\end{array}
\right\vert .
\]
Note that $\left(  t^{\left(  \lambda_{s_{i}}-\lambda_{s_{0}}\right)
\xi_{r_{j}}}\right)  _{0\leq i\leq L-1,0\leq j\leq L-1}$ is a generalized
Vandermonde matrix. As such, \cite[Page 135]{olshevsky2001structured},
\begin{align*}
P_{S}\left(  t\right)   &  =\left(
{\displaystyle\prod\limits_{i>j}}
\left(  t^{\lambda_{s_{i}}-\lambda_{s_{0}}}-t^{\lambda_{s_{j}}-\lambda_{s_{0}%
}}\right)  \right)  \times\left(
{\displaystyle\prod\limits_{j=0}^{L-1}}
t^{\lambda_{s_{0}}\xi_{r_{j}}}\right)  \\
&  \times s_{\left(  \xi_{r_{L-1}}-L+1,\xi_{r_{L-2}}-L+2,\cdots,\xi_{r_{0}%
}\right)  }\left(  t^{\lambda_{s_{0}}-\lambda_{s_{0}}},t^{\lambda_{s_{1}%
}-\lambda_{s_{0}}},\cdots,t^{\lambda_{s_{L-1}}-\lambda_{s_{0}}}\right)
\end{align*}
where $s_{\left(  \xi_{r_{L-1}}-L+1,\xi_{r_{L-2}}-L+2,\cdots,\xi_{r_{0}%
}\right)  }$ is the Schur function with positive integer coefficients where
$\left(  \xi_{r_{L-1}}-L+1,\xi_{r_{L-2}}-L+2,\cdots,\xi_{r_{0}}\right)  $ is a
partition of the positive integer $\left(  \xi_{r_{L-1}}-L+1\right)  +\left(
\xi_{r_{L-2}}-L+2\right)  +\cdots+\xi_{r_{0}}.$ In other words, there exists a
finite sequence $I$ of $L$-tuples of non-negative integers such that \
\begin{align*}
&  s_{\left(  \xi_{r_{L-1}}-L+1,\xi_{r_{L-2}}-L+2,\cdots,\xi_{r_{0}}\right)
}\left(  t^{\lambda_{s_{0}}-\lambda_{s_{0}}},t^{\lambda_{s_{1}}-\lambda
_{s_{0}}},\cdots,t^{\lambda_{s_{L-1}}-\lambda_{s_{0}}}\right)  \\
&  =\sum_{\left(  m_{0},m_{1}\cdots,m_{L-1}\right)  \in I}t^{m_{0}\left(
\lambda_{s_{0}}-\lambda_{s_{0}}\right)  +m_{1}\left(  \lambda_{s_{1}}%
-\lambda_{s_{0}}\right)  +\cdots+m_{L-1}\left(  \lambda_{s_{L-1}}%
-\lambda_{s_{0}}\right)  }%
\end{align*}
where each integer $m_{j}$ represents the number of times the $j$-th variable
of the Schur function, occurs in the $\left(  \xi_{r_{L-1}}-L+1,\xi_{r_{L-2}%
}-L+2,\cdots,\xi_{r_{0}}\right)  $-tableaux. Thus,
\begin{align*}
P_{S}\left(  t\right)   &  =\left(
{\displaystyle\prod\limits_{j=0}^{L-1}}
t^{\lambda_{s_{0}}\xi_{r_{j}}}\right)  \times%
{\displaystyle\prod\limits_{i>j}}
\left(  t^{\lambda_{s_{i}}-\lambda_{s_{0}}}-t^{\lambda_{s_{j}}-\lambda_{s_{0}%
}}\right)  \\
&  \times\left(  \sum_{\left(  m_{0},m_{1}\cdots,m_{L-1}\right)  \in
I}t^{m_{0}\left(  \lambda_{s_{0}}-\lambda_{s_{0}}\right)  +m_{1}\left(
\lambda_{s_{1}}-\lambda_{s_{0}}\right)  +\cdots+m_{L-1}\left(  \lambda
_{s_{L-1}}-\lambda_{s_{0}}\right)  }\right)
\end{align*}
and in light of the discussion above, given a positive real number $r,$ it is
clear that $P_{S}\left(  r\right)  $ is not equal to zero since
\[
\sum_{\left(  m_{0},m_{1}\cdots,m_{L-1}\right)  \in I}r^{m_{0}\left(
\lambda_{s_{0}}-\lambda_{s_{0}}\right)  +m_{1}\left(  \lambda_{s_{1}}%
-\lambda_{s_{0}}\right)  +\cdots+m_{L-1}\left(  \lambda_{s_{L-1}}%
-\lambda_{s_{0}}\right)  }>0
\]
and we conclude that $P_{S}$ is a nonzero polynomial in the variable $t.$
Next, since
\[
\left\vert \det\left(  S\right)  \right\vert =\left\vert \det\left(  \left(
e^{2\pi i\tau\lambda_{s_{k}}\xi_{r_{j}}}\right)  _{0\leq j\leq L-1,0\leq k\leq
L-1}\right)  \right\vert =\left\vert P_{S}\left(  e^{2\pi i\tau}\right)
\right\vert
\]
under the assumption that $e^{2\pi i\tau}$ is a transcendental number, it
follows that $\det\left(  S\right)  =P_{S}\left(  e^{2\pi i\tau}\right)
\neq0.$ Otherwise, let $Q$ be the least common multiple of all polynomials
$P_{S}$ where $S$ runs through all possible square submatrices of
$\mathcal{M}.$ Next, let $\tau$ be a real number such that $e^{2\pi i\tau}$ is
an algebraic complex number which is not a root of the polynomial $Q.$ Fixing
such a real number $\tau$, we obtain that every minor of $\mathcal{M}_{\tau}$
is not equal to zero.

\begin{example}
	Let
	\[
	D=\left(
	\begin{array}
	[c]{cccc}%
	e^{2\pi i\sqrt{2}} & 0 & 0 & 0\\
	0 & e^{4\pi i\sqrt{2}} & 0 & 0\\
	0 & 0 & e^{6\pi i\sqrt{2}} & 0\\
	0 & 0 & 0 & e^{8\pi i\sqrt{2}}%
	\end{array}
	\right)  ,T=\left(
	\begin{array}
	[c]{cccc}%
	0 & 0 & 0 & 1\\
	1 & 0 & 0 & 0\\
	0 & 1 & 0 & 0\\
	0 & 0 & 1 & 0
	\end{array}
	\right)  ,f=\left(
	\begin{array}
	[c]{c}%
	1\\
	2\\
	3\\
	4
	\end{array}
	\right)  .
	\]
	With the help of a Computer Algebra System, we were able to verify that $$\left\{
	D^{k}T^{j}f:j\in\mathbb{Z}_{4},k\in\left\{  0,1,\cdots,6\right\}  \right\}$$
	is a full spark frame with upper and lower frame bounds roughly around $258$
	and $174$ respectively.
\end{example}



\end{document}